\documentclass{article}
\usepackage{amssymb,amsmath}
\usepackage{graphicx,enumerate}

\title{Flat topology on prime, maximal and minimal prime spectra of quantales}

\author{George Georgescu \\ \footnotesize University of Bucharest\\ \footnotesize Faculty of Mathematics and Computer Science\\ \footnotesize Bucharest, Romania\\ \footnotesize Email: georgescu.capreni@yahoo.com}

\date{}

\begin{document}
\maketitle

\begin{abstract}
Several topologies can be defined on the prime, the maximal and the minimal prime spectra of a commutative ring; among them, we mention the Zariski topology, the patch topology and the flat topology. By using these topologies, Tarizadeh and Aghajani obtained recently new characterizations of various classes of rings: Gelfand rings, clean rings, absolutely flat rings, $mp$ - rings,etc. The aim of this paper is to generalize some of their results to quantales, structures that constitute a good abstractization for lattices of ideals, filters and congruences. We shall study the flat and the patch topologies on the prime, the maximal and the minimal prime spectra of a coherent quantale. By using these two topologies one obtains new characterization theorems for hyperarchimedean quantales, normal quantales, B-normal quantales, $mp$ - quantales and $PF$ - quantales. The general results can be applied to several concrete algebras: commutative rings, bounded distributive lattices, MV-algebras, BL-algebras, residuated lattices, commutative unital $l$ - groups, etc.

\end{abstract}

\textbf{Keywords}: flat topology, coherent quantale, reticulation, hyperarchimedean quantales, normal and $B$ - normal quantales, $mp$ - quantales.
\newtheorem{definitie}{Definition}[section]
\newtheorem{propozitie}[definitie]{Proposition}
\newtheorem{remarca}[definitie]{Remark}
\newtheorem{exemplu}[definitie]{Example}
\newtheorem{intrebare}[definitie]{Open question}
\newtheorem{lema}[definitie]{Lemma}
\newtheorem{teorema}[definitie]{Theorem}
\newtheorem{corolar}[definitie]{Corollary}

\newenvironment{proof}{\noindent\textbf{Proof.}}{\hfill\rule{2mm}{2mm}\vspace*{5mm}}

\section{Introduction}

 \hspace{0.5cm}The flat topology on the prime spectrum of a ring was introduced by Hochster in \cite{Hochster} under the name of inverse topology. It was rediscovered by Doobs et al. in \cite{Doobs}, where the terminology of "flat topology" appeared. The flat topology is strongly related to other two topologies defined on the prime spectrum of a commutative ring: spectral topology and patch topology \cite{Dickmann},\cite{Johnstone},\cite{Tar1}. These three topologies have a deep impact on some important themes in ring theory (see \cite{Hochster},\cite{Dickmann},\cite{Aghajani}). Recently, Tarizadeh proposed in \cite{Tar1} purely algebraic definitions for flat and patch topologies on the prime spectrum $Spec(R)$ of a commutative ring $R$. For examples, the closed sets in the flat topology on $Spec(R)$ are defined as the images $Im(f^{\ast})$, where $f^{\ast} : Spec(A)  \rightarrow Spec(R)$ is the map induced by a flat morphism $f:R \rightarrow   A$.
Further, the flat and the patch topologies were used to obtain new properties and new characterizations of Gelfand rings and clean rings, as well as of new classes of rings: mp-rings and purified rings (see \cite{Aghajani},\cite{Tar2},\cite{Tar3}).
A natural problem is to use the flat and the patch topology for obtaining new results on the spectra of other types of algebras.
On the other hand, the quantales constitute a good abstraction of the lattices of ideals, filters and congruences in various algebraic structures. A rich literature was dedicated to the spectra of quantales (see \cite{Eklud},\cite{Georgescu},\cite{PasekaRN},\cite{Paseka},\cite{Rosenthal}). Besides the quantales, other types of multiplicative lattices were proposed to abstractize the lattices of ideals, filters and congruences \cite{GeorgescuVoiculescu}, \cite{Martinez},\cite{Simmons}, \cite{SimmonsC}.
In this paper we shall study the flat topology on prime, maximal and minimal prime spectra of a coherent quantale in connection to important classes of quantales: hyperarchimedean, normal, $B$-normal, $mp$ - quantales, etc. We shall obtain new characterizations of these types of quantales in terms of some algebraic and topological properties of spectra. Our results extend some theorems proven in \cite{Aghajani},\cite{Tar1},\cite{Tar2},\cite{Tar3}, \cite{Al-Ezeh1}, \cite{Al-Ezeh2}, etc.  for the spectra of commutative rings. The proofs of some results will use the reticulation of a quantale, a construction that assigns to each coherent quantale $A$ a unique bounded distributive lattice $L(A)$, whose prime spectrum $Spec_{Id}(L(A))$ is homeomorphic to the prime spectrum of $A$.

Now we shall describe the content of this paper. In Section $2$ we present some notions and basic results on the prime and the maximal spectra of a quantale, their spectral topologies and the radical elements (cf. \cite{Rosenthal},\cite{Martinez},\cite{Eklud},\cite{Paseka}).

Section $3$ contains the axiomatic definition of reticulation of a coherent quantale $A$ and the connection between the $m$-prime elements of $A$ and the prime ideals of the reticulation $L(A)$. The homeomorphism between the prime spectrum $Spec(A)$ of $A$ and the space $Spec_{Id}(L(A))$ of the prime ideals in $L(A)$ (with the Stone topology) induces on $Spec(A)$ a spectral topology (this spectral space is denoted by $Spec_Z(A)$, because it generalizes the Zariski topology).

Section $4$ deals with the Boolean center $B(A)$ of a quantale $A$. We present a short proof that $B(A)$ is isomorphic to the Boolean algebra $B(L(A))$ of  complemented elements in $L(A)$. Section $5$ concerns the patch and the flat topology on $Spec(A)$, associated with the spectral space $Spec_Z(A)$;these two topological spaces will be denoted by $Spec_P(A)$, resp. $Spec_F(A)$. We also introduce the Pierce spectrum $Sp(A)$ of the quantale $A$. In Section $6$ we continue the study of the hyperarchimedean quantales, initiated in \cite{Cheptea1}. The main result of the section presents new characterizations of these objects in terms of the flat and the patch topologies.

Section $7$ contains results about the flat topology on the maximal spectrum $Max(A)$ of the coherent quantale $A$. The new topological space $Max_F(A)$ is Hausdorff and zero - dimensional. We use the flat topology on $Max(A)$ in order to obtain new properties that characterize  the normal and the B-normal quantales \cite{Cheptea1},\cite{GeorgescuVoiculescu2},\cite{PasekaRN},\cite{SimmonsC}. The normal quantales constitute an abstractization of the lattices of ideals in Gelfand rings \cite{Johnstone},\cite{Lam},\cite{c} and in normal lattices\cite{Cornish},\cite{Pawar},\cite{GeorgescuVoiculescu}, while the B-normal quantales generalize the lattices of ideals in clean rings \cite{a},\cite{b} and in B-normal lattices \cite{Cignoli},\cite{Cheptea}. Thus our theorems on normal and B-normal quantales can be applied to various types of structures: Gelfand rings and normal lattices, normal and $B$ - normal lattices, Gelfand residuated lattices and residuated lattices with Boolean lifting property \cite{GCM}, clean unital $l$ - groups \cite{f},etc.

In Section $8$ we study two topologies on the set $Min(A)$ of the minimal  $m$ - prime elements of a coherent quantale $A$. Thus we obtain two topological spaces $Min_Z(A)$ and $Min_F(A)$ : the first one is a subspace of $SpecZ(A)$ and the second one is a subspace of $Spec_F(A)$. By using the reticulation we prove that  $Min_Z(A)$ is a zero - dimensional Hausdorff space and $Min_F(A)$ is a compact $T1$ space. The main results of this section focus on the $mp$ - quantales, a notion that generalizes the $mc$ - rings of \cite{Aghajani}. We present several conditions that characterize $mp$ - quantales. For example, we prove that a coherent quantale $A$ is an $mc$ - quantale iff the reticulation $L(A)$ is a conormal lattice \cite{Simmons} iff $Spec_F(A)$ is a normal space. We introduce the $PF$ - quantales as a generalization of the $PF$ - rings \cite{Al-Ezeh2} and we prove that a coherent quantale $A$ is a $PF$ - quantale if and only if it is a semiprime $mp$ - quantale. Further we obtain a characterization theorem for $PF$ - quantales.

\section{Preliminaries}

 \hspace{0.5cm} Let $(A,\lor, \land, \cdot, 0, 1  )$ be a quantale and $K(A)$ the set of its compact elements. $A$ is said to be integral if $(A, \cdot, 1)$ is a monoid and commutative, if the multiplication $\cdot$ is commutative. A frame is a quantale in which the multiplication coincides with the meet \cite{Johnstone}. The quantale $A$ is algebraic if any $a \in A$ has the form $a= \bigvee  X$ for some subset $X$ of $K(A)$. An algebraic quantale $A$ is coherent if $1 \in K(A)$ and $K(A)$ is closed under the multiplication.
Throughout this paper, the quantales are assumed to be integral and commutative. Often we shall write $ab$ instead of $a \cdot b$. We fix a quantale $A$.

\begin{lema}
\cite{Birkhoff} For all elements $a, b, c$ of the quantale $A$ the following hold:
\newcounter{nr}
\begin{list}{(\arabic{nr})}{\usecounter{nr}}
\item If $a\lor b =1$ then $a\cdot b= a \land b $;
\item If $a\lor b =1$ then $a^{n}\lor b^{n}=1$ for all integer numbers $n\geq 1$;
\item If $a\lor b=a\lor c=1$ then $a\lor (b\cdot c)= a \lor (b\land c)=1$;
\item If $a\lor b =1$ and $a\leq c$ then $a\lor (b\cdot c)=c$.
\end{list}
\end{lema}

One can define on the quantale $A$ a residuation operation $a \rightarrow b=\bigvee \{ x|a x \leq b\}$ and a negation operation $a^{\bot } =a \rightarrow 0 =\bigvee \{x |ax=0\}$. Thus $(A, \lor, \land, \cdot, \rightarrow, 0, 1)$ is a residuation lattice \cite{Galatos}, \cite{Kowalski}. In this paper we shall use without mention the basic arithmetical properties of a residuated lattice.

An element $p <1$ of $A$ is $m$-{\emph{prime}} if for all $a, b \in A$, $ab \leq p$ implies $a \leq b$ or $b \leq p$. If $A$ is an algebraic quantale, then $p<1$ is $m$-prime if and only if for all $c, d \in K(A)$, $cd \leq p$ implies $c \leq p$ or $d \leq p$. Let us introduce the following notations: $Spec(A)$ is the set of $m$-prime elements and $Max(A)$ is the set of maximal elements of $A$. If $1 \in K(A)$ then for any $a <1$ there exists $m \in Mar(A)$ such that
$a \leq m$. The same hypothesis $1 \in K(A)$ implies that $Max(A) \subseteq Spec(A)$.

Let $R$ be a (unital) commutative ring and $L$ a bounded distributive lattice. Let us denote by $Id(R)$ the quantale of ideals in $R$ and by $Id(L)$ the frame of ideals in $L$. Thus the set $Spec(R)$ of prime ideals in $R$ is the prime spectrum of the quantale $Id(R)$ and the set of prime ideals in $L$ is the prime spectrum of the frame $Id(L)$.

The {\emph{radical}} $\rho(a)=\rho_A(a)$ of an element $a \in A$ is defined by $\rho_A(a)=\bigwedge \{p\in Spec(A)|a \leq p\}$; if $a=\rho(a)$ then $a$ is a radical element. We shall denote by $R(A)$ the set of radical elements of $A$. The quantale is {\emph{semiprime}} if $\rho(0)=0$.

\begin{lema}
\cite{Rosenthal} For all elements $a, b \in A$ the following hold:
\usecounter{nr}
\begin{list}{(\arabic{nr})}{\usecounter{nr}}
\item $a \leq \rho(a)$;
\item $\rho(a \land b)=\rho (ab)=\rho(a) \land \rho(b)$;
\item $\rho(a)=1$ iff $a=1$;
\item $\rho(a \lor b)=\rho(\rho(a) \lor \rho(b))$;
\item $\rho(\rho(a))=\rho(a)$;
\item $\rho(a) \lor \rho(b)=1$ iff $a \lor b=1$;
\item $\rho(a^n)=\rho(a)$, for all integer $n \geq 1$.
\end{list}
\end{lema}

For an arbitrary family $(a_i)_{i\in I} \subseteq A$, the following equality holds: $\rho(\displaystyle \bigvee_{i \in I}a_i)=\rho(\bigvee_{i \in I} \rho(a_i))$. If $(a_i)_{i\in I} \subseteq R(A)$ then we denote $\displaystyle \bigvee_{i \in I}^{\cdot} a_i=\rho(\bigvee_{i \in I}a_i)$. Thus it easy to prove that $(R(A), \displaystyle \bigvee^{\cdot}, \wedge, \rho(a), 1)$ is a frame \cite{Rosenthal}.

\begin{lema}
\cite{Cheptea1} If $1 \in K(A)$ then $Spec(A)=Spec(R(A))$ and $Max(A)=Max(R(A))$.
\end{lema}

\begin{lema}
\cite{Martinez} Let $A$ be a coherent quantale and $a \in A$. Then
\usecounter{nr}
\begin{list}{(\arabic{nr})}{\usecounter{nr}}
\item $\rho(a)=\bigvee \{c \in K(A)| c^k \leq a$ for some integer $k\geq 1\}$;
\item For any $c \in K(A), c \leq \rho(a)$ iff $c^k \leq a$ for some $k\geq 1$.
\end{list}
\end{lema}

\begin{lema}
\cite{Cheptea1} If $A$ is a coherent quantale then $K(R(A))=\rho(K(A))$ and $R(A)$ is a coherent frame.
\end{lema}

For any element $a$ of a coherent quantale $A$ let us consider the interval $[a)_A=\{x \in A|a \leq x\}$ and for all $x, y \in [a)_A$ denote $x \cdot_a y=xy \lor a$. Thus $[a)_A$ is closed under the multiplication $\cdot_a$ and $([a)_A, \lor, \land, \cdot_a, 0, 1)$ is a coherent quantale.

\begin{lema}
\cite{Cheptea1} The quantale $([\rho(a))_A, \lor, \land, \cdot_a, 0, 1)$ is semiprime and $Spec(A)=Spec([\rho(a))_A), Max(A)=Max([\rho(a))_A)$.
\end{lema}

Let $A, B$ be two quantales. A function $f: A \rightarrow B$ is a morphism of quantales if it preserves the arbitrary joins and the multiplication; $f$ is an integral morphism if $f(1)=1$.

\begin{lema}
\cite{Cheptea1} Let $A$ be a coherent quantale and $a \in A$.
\usecounter{nr}
\begin{list}{(\arabic{nr})}{\usecounter{nr}}
\item The function $u_a^A : A \rightarrow [a)_A$, defined by $u_a^A(x)=x \lor a$, for all $x \in A$, is an integral quantale morphism;
\item If $c \in K(A)$ then $u_a^A(c) \in K([a))$.
\end{list}
\end{lema}

Let $A$ be a quantale such that $1 \in K(A)$. For any $a \in A$, denote $D(a)=\{p \in Spec(A)|a \not\leq p\}$ and $V(a)=\{p \in Spec(A)|a \leq p\}$. Then
$Spec(A)$ is endowed with a topology whose closed sets are $(V(a))_{a \in A}$. If the quantale $A$ is algebraic then the family $(D(c))_{c \in K(A)}$ is a basis of open sets for this topology. The topology introduced here generalizes the Zariski topology (defined on the prim spectrum $Spec(R)$ of a commutative ring $R$ \cite{Atiyah}) and the Stone topology (defined on the prime spectrum $Spec_{Id}(L)$ of a bounded distributive lattice $L$ \cite{BalbesDwinger}).

Thus we denote by $Spec_Z(A)$ the prime spectrum $Spec(A)$ endowed with the above defined topology; $Max_Z(A)$ will denote the maximal spectrum $Max(A)$ considered as a subspace of $Spec_Z(A)$.

Let $L$ be a bounded distributive lattice. For any $x \in L$, denote $D_{Id}(x)=\{P \in Spec_{Id}(L)|x \not \in P\}$ and $V_{Id}(x)=\{P \in Spec_{Id,Z}(L)|x \in P\}$. The family $(D_{Id}(x))_{x \in L}$ is a basis of open sets for the Stone topology on $Spec_{Id}(L)$; this topological space will be denoted by $Spec_{Id,Z}(L)$. Let $Max_{Id}(L)$ be the set of maximal ideals of $L$. Thus
$Max_{Id}(L) \subseteq Spec_{Id}(L)$ and $Max_{Id}(L)$ becomes a subspace of $Spec_{Id}(L)$, denoted $Max_{Id,Z}(L)$.

\section{Reticulation of a coherent quantale}

 \hspace{0.5cm} In this section we shall recall from \cite{Cheptea1},\cite{Georgescu} the axiomatic definition of the reticulation of the coherent quantale and some of its basic properties. Let $A$ be a coherent quantale and $K(A)$ the set of its compact elements.

\begin{definitie}
\cite{Cheptea1}
A reticulation of the quantale $A$ is a bounded distributive lattice $L$ together a surjective function $\lambda:K(A)\rightarrow L$ such that for all $a,b\in K(A)$ the following properties hold
\usecounter{nr}
\begin{list}{(\arabic{nr})}{\usecounter{nr}}
\item $\lambda(a\vee b)\leq \lambda(a)\vee \lambda(b)$;
\item $\lambda(ab)= \lambda (a)\wedge \lambda (b)$;
\item $\lambda(a)\leq\lambda(b)$ iff $a^n\leq b$ , for some integer $n\geq 1$.
\end{list}
\end{definitie}

In \cite{Cheptea1},\cite{Georgescu} there were proven the existence and the unicity of the reticulation for each coherent quantale $A$; this unique reticulation will be denoted by $(L(A),\lambda_A:K(A)\rightarrow L(A))$ or shortly $L(A)$. The reticulation $L(R)$ of a commutative ring $R$ there was introduced by many authors, but the main references on this topic remain \cite{Simmons}, \cite{Johnstone}. We remark that L(R) is isomorphic to the reticulation L(Id(R)) of the quantale $Id(R)$.

\begin{lema}
\cite{Cheptea1} For all elements $a, b \in K(A)$ the following properties hold:
\usecounter{nr}
\begin{list}{(\arabic{nr})}{\usecounter{nr}}
\item $a \leq b$ implies $\lambda_A(a)\leq\lambda_A(b)$;
\item $\lambda_A(a \lor b)=\lambda_A(a) \lor \lambda_A(b)$;
\item $\lambda_A(a)=1$ iff $a=1$;
\item $\lambda_A(0)= 0$;
\item $\lambda_A(a)=0$ iff $a^n = 0$, for some integer $n\geq 1$;
\item $\lambda_A (a^n)=\lambda_A(a)$, for all integer $n\geq 1$;
\item $\rho(a)= \rho(b)$ iff $\lambda_A(a)= \lambda_A(b)$;
\item $\lambda_A(a)=0$ iff $a\leq \rho(0)$;
\item If $A$ is semiprime then $\lambda_A(a)= 0$ implies $a= 0$.
\end{list}
\end{lema}
For any $a\in A$ and $I\in Id(L(A))$ let us denote $a^{\ast}= \{\lambda_A(c)|c\in K(A), c\leq a\}$ and $ I_{\ast} = \bigvee\{c\in K(A)|\lambda_A(c)\in I\}$.

\begin{lema}
\cite{Cheptea1} The following assertions hold
\usecounter{nr}
\begin{list}{(\arabic{nr})}{\usecounter{nr}}
\item If $a\in A $ then $a^{\ast}$ is an ideal of $L(A)$ and $a\leq (a^{\ast})_{\ast}$;
\item If $I\in Id(L(A))$ then $(I_{\ast})^{\ast}=I$;
\item If $p\in Spec(A)$ then  $(p^{\ast})_{\ast}= p$ and $p^{\ast}\in Spec_{Id}(L(A))$;
\item If $P\in Spec_{Id}((L(A))$ then $P_{\ast}\in Spec(A)$;
\item If $p\in K(A)$ then $c^{\ast}= (\lambda_A(c)]$.
\end{list}
\end{lema}\
\begin{lema}
\cite{Cheptea1} If $a\in A$ and $I\in Id(L(A))$ then $\rho(a)=(a^{\ast})_{\ast}$, $a^{\ast}= (\rho(a))^{\ast}$ and $\rho(I_{\ast})= I_{\ast}$.
\end{lema}
\begin{lema}
If $c\in K(A)$ and $I\in Id(L(A))$ then $c\leq I_{\ast}$ iff $\lambda_A(c)\in I$.
\end{lema}
\begin{proof}
If $c\leq \bigvee\{d\in K(A))|\lambda_A(d)\in I\}$ then there exists $d\in K(A)$ such that $\lambda_A(c)\in I$ and $c\leq d$. Thus $\lambda_A(c)\leq\lambda_A(d)$, so $\lambda_A(c)\in I$. The converse implication is obvious.
\end{proof}
\begin{lema}
Assume that $c\in K(A)$ and $p\in Spec(A)$. Then $c\leq p$ iff $\lambda_A(c) \in p^{\ast}$.
\end{lema}
According to Lemma 3.3, one can consider the following order- preserving functions: $ u:Spec(A)\rightarrow Spec_{Id}(L(A))$ and $ v:Spec_{Id}(L(A))\rightarrow Spec(A)$, defined by $u(p) = p^{\ast}$ and $v(P) = P_{\ast}$, for all $p\in Spec(A)$ and $P\in Spec_{Id}(L(A))$. Sometimes the previous functions $u$ and $v$ will be denoted by $u_A$ and $v_A$.
\begin{lema}
\cite{Cheptea1}
The functions $u$ and $v$ are homeomorphisms, inverse to one another.
\end{lema}
\begin{corolar}
$Max_Z(A)$ and $Max_{Id,Z}(L(A))$ are homeomorphic.
\end{corolar}
\begin{propozitie}
\cite{Cheptea1}
The functions $\Phi:R(A) \rightarrow Id(L(A))$ and $\Psi:Id(L(A))\rightarrow R(A)$ defined by $\Phi(a)= a^{\ast}$ and $\Psi(I)= I_{\ast}$, for all $a\in R(A)$ and $I \in Id(L(A))$, are frame isomorphisms, inverse to one another.
\end{propozitie}

\begin{corolar} If $I, J$ are ideals of $L(A)$ then $(I\lor J)_{\ast} = \rho(I_{\ast}\lor J_{\ast})$.

\end{corolar}

\begin{proof} The equality $(I\lor J)_{\ast} = \rho(I_{\ast}\lor J_{\ast})$ follows from the fact that the frame isomorphism $\Psi$ preserves the finite joins.

\end{proof}

\section{Boolean center of a quantale and the reticulation}

 \hspace{0.5cm}The Boolean center of a quantale $A$ is the Boolean algebra $B(A)$ of complemented elements of $A$ (cf. \cite{Birkhoff},\cite{Jipsen}). In this section we shall
study some basic properties of the Boolean center of a coherent quantale versus the reticulation.
\begin{lema}
\cite{Birkhoff},\cite{Jipsen} Let $A$ be a quantale and $a,b \in A$, $e\in B(A)$. Then the following properties hold:
\usecounter{nr}
\begin{list}{(\arabic{nr})}{\usecounter{nr}}
\item $ a\in B(A)$ iff $a \lor a^{\bot} = 1$;
\item $ a\land b$ = $ae$;
\item $e\rightarrow a$ = $e^{\bot}\lor a$;
\item If $a\lor b = 1$ and  $ab = 0$, then $a,b \in B(A)$;
\item $(a \land b) \lor e = (a\lor e)\land (b\land e)$;
\item  For any integer $n\geq1$, $a\lor b = 1$ and $a^n b^n = 0$ implies $a^n,b^n \in B(A)$.
\end{list}
\end{lema}
\begin{proof}
The properties (1)-(5) are taken from \cite{Birkhoff},\cite{Jipsen} and (6) follows by (4) and Lemma 2.1,(ii).
\end{proof}
\begin{lema}
\cite{Cheptea1} If $ 1\in K(A)$ then $ B(A)\subseteq K(A)$.
\end{lema}
For a bounded distributive lattice $L$ we shall denote by $B(L)$ the Boolean algebra of the complemented elements of $L$. It is well-known that $B(L)$ is isomorphic  to the Boolean center $B(Id(L))$ of the frame $Id(L)$ (see \cite{Birkhoff}, \cite{Johnstone}, \cite{Cheptea}).

Let us fix a coherent quantale $A$.

\begin{lema}
Assume $c\in K(A)$. Then $\lambda_A(c) \in B(L(A))$ if and only if $c^n \in B(A)$, for some integer $n\geq 1$.
\end{lema}
\begin{proof}
Assume $\lambda_A(c)\in B(L(A))$, hence $\lambda_A(c)\lor \lambda_A(d) = 1$ and $\lambda_A(c)\land \lambda_A(d) = 0$, for some $d\in K(A)$. Then $\lambda_A(c\lor d) = 1$ and $\lambda_A(cd) = 0$, hence, by Lemma 3.2,(2) and (5), it follows that $c \lor d = 1$ and $c^n d^n = 0$, for some integer $n\geq 1$. Therefore by Lemma 4.1,(6) one gets $c^n,d^n \in B(A)$. Conversely, if $c^n \in B(A)$ then $\lambda_A(c) = \lambda_A(c^n)$ is an element of $B(L(A))$.
\end{proof}
\begin{corolar}
\cite{Cheptea1} The function $\lambda_A|_{B(A)} :B(A)\rightarrow B(L(A))$ is a Boolean isomorphism.
\end{corolar}
\begin{proof}
It is easy to see that the function $\lambda_A|_{B(A)} :B(A)\rightarrow B(L(A))$ is an injective Boolean morphism. The surjectivity follows by using Lemma 4.3.
\end{proof}

If $L$ is bounded distributive lattice and $I\in Id(L)$ then the annihilator of $I$ is the ideal $Ann(I) = \{ x \in L |x \land y =0$, for all $y\in L \}$.

The next two propositions concern the behaviour of reticulation w.r.t. the annihilators.
\begin{propozitie}
If $a$ is an element of a coherent quantale then $Ann(a^\ast)=(a \rightarrow \rho(0))^\ast$; if $A$ is semiprime then $Ann(a^\ast)=(a^\perp)^\ast$.
\end{propozitie}

\begin{proof}
Assume $x \in Ann(a^\ast)$, so $x=\lambda_A(c)$ for some $c \in K(A)$ with the property that for all $d \in K(A)$, $d \leq a$ implies $\lambda_A(cd)=\lambda_A(c) \land \lambda_A(d)=0$. By Lemma 3.2(8) one gets $cd \leq \rho(0)$, so $c \leq d \rightarrow \rho(0)$. Thus the following hold: $c \leq \bigwedge \{ d \rightarrow \rho(0)| d \in K(A), d \leq a\}=(\bigvee \{ d\in K(A)|d \leq a\})\rightarrow \rho(0)=a \rightarrow \rho(0)$,
hence $x=\lambda_A(c) \in (a \rightarrow \rho(0))^\ast$. We conclude that $Ann(a^\ast)\subseteq (a \rightarrow \rho(0))^\ast$.

In order to prove that $(a \rightarrow \rho(0))^\ast\subseteq Ann(a^\ast)$ assume that $x\in (a \rightarrow \rho(0))^\ast$, so $x = \lambda_A(c)$ for some $c\in K(A)$ such that  $c \leq a \rightarrow \rho(0)$. For all $d\in K(A)$ with $d\leq a$ we have $c \leq a \rightarrow \rho(0)\leq d \rightarrow \rho(0)$. By Lemma 3.2,(8) one gets $\lambda_A(c)\land \lambda(d)$ = $\lambda_A(cd) = 0$, so $x = \lambda_A(c)\in Ann(a^\ast)$.

\end{proof}

\begin{propozitie}
Assume that $A$ is a coherent quantale. If $I$ is an ideal of $L(A)$ then $(Ann(I))_\ast=I_\ast \rightarrow \rho(0)$; if $A$ is semiprime then
$(Ann(I))_\ast=(I_\ast)^\perp$.
\end{propozitie}

\begin{proof}
In order to verify that $I_\ast \rightarrow \rho(0) \leq (Ann(I))_\ast$, it suffices to show that for all $c \in K(A)$, $c \leq I_\ast \rightarrow \rho(0)$ implies $c \leq (Ann(I))_\ast$. If $c \leq I_\ast \rightarrow \rho(0)$ then

$\bigvee \{cd |d\in K(A), \lambda_A(d) \in I\}=c( \bigvee \{d\in K(A)|\lambda_A(d) \in I\})=c I_\ast \leq \rho(0)$.

Thus for all $d \in K(A)$ with $\lambda_A(d) \in I$ we have $cd \leq \rho(0)$ so $c^n d^n=0$ for some integer $n \geq 1$ (cf. Lemma 2.4 (ii)). It follows that
$\lambda_A(c) \land \lambda_A(d) =\lambda_A(c^n d^n)=0$, hence $\lambda_A(c) \in Ann(I)$, i.e. $c \leq (Ann(I))_\ast$.

Assume now that $c \in K(A)$ and $c \leq (Ann(I))_\ast$, hence by Lemma 3.5, $\lambda_A(c) \in Ann(I)$. For any $c \in K(A)$ with $\lambda_A(d) \in I$ we have $\lambda_A(cd)=\lambda_A(c) \land \lambda_A(d)=0$, hence by Lemma 3.2 (8) one gets $cd\leq \rho(0)$. Therefore we have $cI_\ast=\bigvee \{cd|d \in K(A), \lambda_A(d) \in I\} \leq \rho(0)$, i.e. $c \leq I_\ast \rightarrow \rho(0)$. Then the inequality $(Ann(I))_\ast \leq I_\ast \rightarrow \rho(0)$ is proven, so the equality $(Ann(I))_\ast=I_\ast \rightarrow\rho(0)$ follows.
\end{proof}

An element $a$ of an arbitrary quantale $A$ is said to be pure (or virginal, in the terminology of  \cite{GeorgescuVoiculescu2}) if for all $c\in K(A)$, $c\leq a$ implies $a \lor c^{\perp} = 1$. The pure elements in a quantale extend the pure ideals of a ring \cite{Lam},\cite{SimmonsC} and the $\sigma$ -ideals of a bounded distributive lattice \cite{Cornish1}, \cite{GeorgescuVoiculescu}. More precisely, an ideal $I$ of bounded distributive lattice $L$ is a $\sigma$ - ideal if for all $x\in I$, we have $I\lor Ann(x) = L$.
\begin{lema}
If an element $a$ of coherent quantale $A$ is pure then $a^{\ast}$ is a $\sigma$- ideal of the reticulation $L(A)$. If moreover $A$ is semiprime then for each $\sigma$ - ideal $J$ of $L(A)$, $J_{\ast}$ is a pure element of $A$.
\end{lema}
\begin{proof}
Assume $x\in a^{\ast}$, so $x = \lambda_A(c)$ for some $c\in K(A)$ with $c\leq a$. Since $a$ is pure, $c\leq a$ implies $a\lor c^{\perp}= 1$, so $d\lor e = 1$ for some $d,e\in K(A)$ with the properties $d\leq a$ and $e\leq c^{\perp}$. Thus $\lambda_A(d)\in a^{\ast}$ and $\lambda_A(c)\land\lambda_A(e) = \lambda_A(ce) = \lambda(0) = 0$, i.e. $\lambda_A(e)\in Ann(\lambda_A(c))$. We observe that $\lambda_A(d)\lor \lambda_A(e) = \lambda_A(d\lor e) = 1$, so $a^{\ast} \lor Ann(\lambda_A(e)) = L(A)$. Thus $a^{\ast}$ is a $\sigma$ - ideal.

Now we assume that $A$ is semiprime and J is $\sigma$ - ideal of $L(A)$. In order to prove that $J_{\ast}$ is a pure element of $A$ let us consider a compact element $c$ of $A$ such that $c\leq J_{\ast}$. By Lemma 3.5 we have $\lambda_A(c)\in J$, hence $J\lor Ann(\lambda_A(c))$ = $L(A)$, so there exist two compact elements $d$ and $e$ of $A$ such that $\lambda_A(d)\in J$, $\lambda_A(e)\in Ann(\lambda_A(c))$ and $\lambda_A(d\lor e)$ = $\lambda_A(d)\lor \lambda_A(e)$ = $1$. According to Lemmas 3.5 and 3.2,(3) we get $d\leq J_{\ast}$ and $d\lor e = 1$. From $\lambda_A(e)\in Ann(\lambda_A(c))$ we infer $\lambda_A(ce)$ = $\lambda_A(c)\land \lambda_A(e)$ = $0$, hence $ce = 0$ (because $A$ is semiprime). Thus $e\leq c^{\perp}$, therefore $1$ = $d\lor e\leq J_{\ast}\lor c^{\perp}$. It follows that $J_{\ast}\lor c^{\perp} = 1$, hence $J$ is a $\sigma$ - ideal of $L(A)$.

\end{proof}

\section{Three topological structures on the prime spectrum}

 \hspace{0.5cm} In this section we shall discuss some basic properties concerning three topologies defined on the prime spectrum $Spec(A)$ of a coherent quantale $A$: spectral topology, flat topology and patch topology.
A topological space $(X, \Omega)$ is said to be spectral \cite{Hochster} ( or coherent in the terminology of \cite{Johnstone}) if it is sober and the family $K(\Omega)$ of compact open sets of $X$ is closed under finite intersections, and forms a basis for the topology. The standard examples of spectral spaces are the prime spectrum $Spec(A)$ of a commutative ring $R$ (with the Zariski topology) and the prime spectrum $Spec_{Id}(L)$ of bounded distributive lattice L (with the Stone topology). If $(X,\Omega)$ is a spectral space then in a standard way (see \cite{Dickmann},\cite{Johnstone}) one can define on $X$ the following two topologies:

$\bullet$ the patch topology, having as basis the family of sets $U\bigcup V$, where $U$ is a compact open set in $X$ and $V$ is the complement of a compact open
set (this topological space is denoted by $X_P$);

$\bullet$ the flat topology, having as basis the family of the complements of compact open sets in $X$ (this topological space is denoted by $X_F$).

\begin{lema}
\cite{Eklud}, \cite{Johnstone}
$X_P$ is a Boolean space and $X_F$ is a spectral space.
\end{lema}
\begin{remarca}
If $L$ is a bounded distributive lattice and $X$ is the spectral space $Spec_{Id}(L)$ then the family $(D_{Id}(x)\bigcap V_{Id}(y))_{x,y\in L}$ is a basis of open sets for $X_P$ and the family $(V_{Id}(y))_{y\in L}$ is a basis of open sets for $X_F$.
\end{remarca}

Let $A$ be a coherent quantale. By Proposition 3.7, $Spec_Z(A)$ is homeomorphic with the spectral space $Spec_{Id}(L(A))$, hence it is a spectral space. The family of open sets in $Spec_Z(A)$ will be denoted by $\mathcal{Z} = \mathcal{Z}_A$. For any subset $S$ of $Spec(A)$, $cl_Z(S)
 = V(\bigcap{S})$ is the closure of $S$ in $Spec_Z(A)$; for all $p\in Spec(A)$, we have $cl_Z(\{p\}) = V(p)$.
 Now we can consider the two topologies associated with the spectral space $X = Spec_Z(A)$: the patch topology and the flat topology. We will denote $Spec_P(A) = X_P$ and $Spec_F(A) = X_F$; $\mathcal{P} = \mathcal{P}_A$ will be the family of open sets
 in $Spec_P(A)$ and $\mathcal{F} = \mathcal{F}_A$ the family of open sets in $Spec_F(A)$.
 \begin{remarca}
 (i) The family $\{D(c)\bigcap V(d)| c,d\in K(A)\}$ is a basis of open sets for $Spec_P(A)$;

 (ii) The family $\{V(c)| c\in K(A)\}$ is a basis of open sets for $Spec_F(A)$.
 \end{remarca}
 \begin{remarca}
 (i) The patch topology on $Spec(A)$ is finer than the spectral and the flat topologies on $Spec(A)$ ( i.e. $\mathcal{Z} \subseteq \mathcal{P}$  and $\mathcal{F}  \subseteq  \mathcal{P}$);

 (ii) The inclusions $\mathcal{Z} \subseteq \mathcal{P}$  and $\mathcal{F}  \subseteq  \mathcal{P}$ show that the identity functions $id: Spec_P(A) \rightarrow Spec_Z(A)$  and  ${id:Spec_P(A) \rightarrow  Spec_F(A)}$ are continuous.
 \end{remarca}
 \begin{propozitie}
 The two inverse functions $u: Spec(A)\rightarrow Spec_{Id}(L(A))$  and $v: Spec_{Id}(L(A))\rightarrow Spec(A)$ from Proposition 3.7 are homeomorphisms w.r.t  the patch and the flat topologies.
 \end{propozitie}

 \begin{proof}
 Applying Lemma 3.6 it is easy to prove that for all $c \in K(A)$, the following equalities $u^{-1}(V_{Id}(\lambda_A(c)))=V(c)$ and $u^{-1}(D_{Id}(\lambda_A(c)))=D(c)$ hold. Therefore, by Remarks 5.2 and 5.3, $u$ is patch and flat continuous.
 \end{proof}

 For any $p\in Spec(A)$, let us denote $\Lambda(p) = \{q\in Spec(A)| q\leq p\}$.

 \begin{propozitie}
 For any  $p\in Spec(A)$, the flat closure of the set $\{p\}$ is $cl_F\{p\}$ = $\Lambda(p)$.
 \end{propozitie}
 \begin{proof}
 According to the definition of the closure $cl_F(p)$ = $cl_F\{p\}$,  the following equalities hold:\\

 $cl_F(p)$ = $\{q\in Spec(A)| \forall c\in K(A) (q\in  V(c) \Rightarrow V(c)\bigcap \{p\} \not= \emptyset \})$

$\hspace{1cm}$ =$\{q\in Spec(A)| \forall c\in K(A) (c\leq q \Rightarrow c\leq p)\}$.\\

In order to prove that $cl_F(p) \subseteq \Lambda(p)$, let us consider $q\in cl_F(p)$ and $c\in K(A)$. Then $c\leq q$ implies $c\leq q$ , therefore

$q=\bigvee \{c \in K(A)| c \leq q\} \leq \bigvee\{c \in K(A)|c \leq p\}=p$.

Conversely, assume that  $q\in \Lambda(p)$, so $q\leq p$.  Thus for any $c\in K(A)$, $c\leq q$  implies $c\leq p$, hence $q\in cl_F(p)$.
  \end{proof}
\begin{propozitie}
If $S\subseteq Spec(A)$  is compact in $Spec_Z(A)$  then its flat closure is $cl_F(S)=\displaystyle \bigcup_{p
\in S} \Lambda(p)$ .
\end{propozitie}
\begin{proof}
Applying Proposition 5.6, it follows that for any $p\in S$ we have $\Lambda(p)$ = $cl_F(p)\subseteq cl_F(S)$, so $\displaystyle \bigcup_{p \in S}\Lambda(p) \subseteq cl_F(S)$. Let us prove the converse inclusion $cl_F(S) \subseteq \displaystyle \bigcup_{p \in S}\Lambda(p)$. Assume by absurdum that there exists $q \in cl_F(S)-\displaystyle \bigcup_{p \in S}\Lambda(p)$, so $q \not\leq p$ for all $p \in S$. Then for all $p \in S$ there exists $c_p \in K(A)$ such that $c_p \not\leq p$ and $c_p \leq q$. This means that $S \subseteq \displaystyle \bigcup_{p \in S} D(c_p) $, so $S \subseteq \displaystyle \bigcup_{i=1}^n D(c_{p_i})$ for some $p_1, \ldots, p_n \in S$. Denote
 $c=\displaystyle \bigvee_{i=1}^n c_{p_i}$, so $c \in K(A)$ and $S \subseteq D(c)$. One remarks that $c \leq q$, so $q \in V(c)$. Since $q\in cl_F(S)$ and $V(c)$ is an open neighbourhoud of $q$ in the flat topology, it follows that $S\bigcap V(c)\not=\emptyset$. This contradicts $S \subseteq D(c)$, hence $cl_F(S) \subseteq \displaystyle \bigcup_{p \in S}\Lambda(p)$. We conclude that $cl_F(S)=\displaystyle \bigcup_{p
\in S} \Lambda(p)$.
\end{proof}

An element $a\in A$ is regular if it is a join of complemented elements. A maximal element in the set of proper regular elements is called max- regular. The set $Sp(A)$ of max- regular elements of $A$ is called the Pierce spectrum of the quantale $A$. For any proper regular element $a$  there exists $p\in Sp(A)$ such that $a\leq p$. If $e\in B(A)$ then we denote $U(e)$ = $\{p\in Sp(A)| e\not\leq a\}$. Thus it is easy to prove that the family $(U(e))_{e\in B(A)}$ is a basis of open sets for a topology on $Sp(A)$.

For any $p\in Spec(A)$ we define $s_A(p)$ = $\bigvee \{e\in B(A)| e\leq p \}$; $s_A(p)$ is regular and $s_A(p) \leq p < 1$.
\begin{lema}
$s_A(p)$ is a max - regular element of $A$.
\end{lema}
\begin{proof}
In order to prove that $s_A(p)\in Sp(A)$ is suffices to have: $e\in B(A)$ and $e\not\leq p$ implies  $s_A(p)\lor e = 1$. Assume by absurdum that there exists  $e\in B(A)$ such that $e\not\leq p$ and $s_A(p)\lor e < 1$. The element $s_A(p)\lor e$  is regular so there exists a max - regular element $q$ such that $s_A(p)\lor e \leq q$. Since $e\not\leq p$ and $p\in Spec(A)$ we have $\neg e \leq p$ , so $1 = e\lor \neg \leq p$. This contradiction shows that $s_A(p)\in Sp(A)$.
\end{proof}

According to the previous lemma, for each $p\in Spec(A)$, $s_A(p)$ is a max - regular element of $A$, so one obtains a function $s_A : Spec(A)\rightarrow  Sp(A)$.
\begin{propozitie}
$Sp(A)$ is a Boolean space and $s_A : Spec(A)\rightarrow  Sp(A)$ is surjective and continuous w.r.t. both flat and spectral topologies on $Spec(A)$.
\end{propozitie}
\begin{proof}
Assume that $q\in Sp(A)$  then $q\leq p$  for some $p\in Spec(A)$, hence $q\leq s_A(p)$.  Since $q$ and $s_A(p)$ are max - regular we have $q = s_A(p)$, so $s_A$ is surjective. It is easy to see that $Sp(A)$ is a Hausdorff space and $(U(e))_{e\in B(A)}$ is basis of clopen set for $Sp(A)$. For all $e\in B(A)$ we have $s^{-1}_A(U(e))
= D(e) = V(\neg
e)$, hence the functions $s_A : Spec_Z(A)\rightarrow  Sp(A)$ and $s_A : Spec_F(A)\rightarrow  Sp(A)$ are continuous. Therefore the topological space $Sp(A) = Im(s_A)$ is compact. We conclude that $Sp(A)$ is a Boolean space.
\end{proof}

\section{Hyperarchimedean  quantales}

  \hspace{0.5cm}The hyperarchimedean quantales were introduced in \cite{Cheptea1}, where a characterization theorem of these objects was proven. This section contains new algebraic and topological characterizations of the hyperarchimedean quantales.
 Let $A$ be a coherent quantale. By \cite{Cheptea1}, $A$ is said to be hyperarchimedean if for any $c\in K(A)$ there exists an integer $n > 1$ such that $c^n\in B(A)$. Applying Lemma 4.2, it follows that a coherent frame $A$ is hyperarchimedean if and only if $B(A) = K(A)$ (see \cite{MartinezZenk}).
\begin{remarca}
Recall from \cite{Banaschewski} that a frame $L$ is zero - dimensional if any element $a\in A$ is a joint of complemented elements. On the other hand, by Remark 2.1,(i) of \cite{MartinezZenk}, an algebraic frame is zero - dimensional if and only if each compact element is complemented. Thus a coherent frame is hyperarchimedean if and only if it is zero - dimensional.
\end{remarca}

\begin{propozitie}
If $A$ is a coherent quantale the the following are equivalent:
\usecounter{nr}
\begin{list}{(\arabic{nr})}{\usecounter{nr}}
\item $A$ is hyperarchimedean;
\item $L(A)$ is a Boolean algebra;
\item $Spec(A) = Max(A)$;
\item The quantale $[\rho(0))_A$ is hyperarchimedean;
\item $R(A)$ is a hyperarchimedean frame;
\item  $R(A)$ is a zero - dimensional frame;
\end{list}

\end{propozitie}

\begin{proof}
The equivalence of (1),(2),(3) and (6) was proven in \cite{Cheptea1}, but for sake of completeness we shall present a short proof of the proposition.

$(1)\Leftrightarrow(2)$ In accordance to Lemma 4.3, the following assertions are equivalent:

$\bullet$ $L(A)$ is a Boolean algebra:

$\bullet$ for all $c\in K(A)$, $\lambda_A(c)\in B(L(A))$;

$\bullet$ for all $c\in K(A)$, there exists an integer $n> 0$  such that $c^n\in B(A)$;

$\bullet$  $A$ is hyperarchimedean.

$(2)\Leftrightarrow(3)$ By the Nachbin theorem \cite {BalbesDwinger} and Proposition 3.7, $L(A)$ is a Boolean algebra iff $Spec_{Id}(L(A)) = Max_{Id}(L(A))$ iff $Spec(A) = Max(A)$.

$(3)\Leftrightarrow(4)$ This equivalence follows by using Lemma 2.6.

$(1)\Leftrightarrow(5)$ According to Lemma 6 of  \cite{Cheptea1},  $Spec(A) = Spec(R(A))$ and $Max(A) = Max(R(A))$, therefore this equivalence follows by using that $(1)$ and $(3)$ are equivalent.

$(5)\Leftrightarrow(6)$ By Remark 6.1.
\end{proof}

\begin{lema}
If $L$ is a bounded distributive lattice then the following are equivalent:
\usecounter{nr}
\begin{list}{(\arabic{nr})}{\usecounter{nr}}
\item $L$ is a Boolean algebra;
\item $Spec_{Id}(L)$ = $Max_{Id}(L)$;
\item For all distinct prime ideals $M$ and $N$ of $L$ there exist $x\notin M$ and $y\notin N$ such that $x \land y = 0$.
\end{list}
\end{lema}

\begin{proof}
$(1)\Leftrightarrow(2)$ By the Nachbin theorem.

$(1)\Rightarrow(3)$ Assume that $L$ is a Boolean algebra and $M, N$ are distinct prime ideals of $L$. Thus  $M, N$ are distinct maximal ideals of $L$ so there exists an element $x\in M - N$. Denoting $y= \neg x$  one gets $x\notin N$, $y\notin M$  and $x\land y = 0$.

$(3)\Rightarrow(1)$ Assume by absurdum that there exist two ideals $M$ and $N$ of $L$ such that $M\not\subseteq N$ . By hypothesis there exist two elements $x$ and $y$ of $L$ such that $x\notin M$, $y\notin N$  and $x\land y = 0$. Since $x\notin M$ and $x\land y = 0$ implies $y\in M$, it follows a contradiction, so $M = N$. It follows that $Spec_{Id}(L)$ = $Max_{Id}(L)$.
\end{proof}

\begin{propozitie}
Assume that $A$ is a coherent quantale. The following following properties are equivalent:
\usecounter{nr}
\begin{list}{(\arabic{nr})}{\usecounter{nr}}
\item For all distinct $p,q\in Spec(L(A))$ there exist $c,d\in K(A)$ such that $c\not\leq p$, $d\not\leq q$ and $cd = 0$;

\item For all distinct prime ideals $I, J$ of $L(A)$ there exist two elements $x, y$ of $L(A)$ such that $x\notin I$, $y\notin J$ and $x\land y= 0$.

\end{list}
\end{propozitie}

\begin{proof}
Assume that $I, J$ are two distinct prime ideals of $L(A)$. In accordance to Proposition 3.7 there exist two  $p,q\in Spec(A)$ such that $I = p^{\ast}$, $J = q^{\ast}$ and $p\not= q$. By hypothesis there exist $c,d\in K(A)$ such that $c\not\leq p$, $d\not\leq q$ and $cd = 0$. Applying Lemmas 3.2 and 3.6 one gets $\lambda_A(c)\land \lambda_A(d)$ = $\lambda_A(cd)$ = $\lambda_A(0)$ = $0$ and $\lambda_A(c)\notin p^{\ast}$, $\lambda_A(d)\notin q^{\ast}$.

Assume now  that $p,q\in Spec(A)$, with $p\neq q$ so  $p^{\ast}$, $q^{\ast}$ are distinct prime ideals of $L(A)$. Thus there exist $c,d\in K(A)$ such that $\lambda_A(c)\notin p^{\ast}$, $\lambda_A(d)\notin q^{\ast}$ and $\lambda_A(cd)$ = $\lambda_A(c)\land \lambda_A(d)$ = 0. Applying Lemma 3.6 one obtains $c\not\leq p$ and $d\not\leq q$. Since $A$ is semiprime,  $\lambda_A(cd) = 0$ implies $cd$ = 0 (by Lemma 3.2,(9)), so there exist an integer $n\geq 1$ such that $c^nd^n = 0$. Let us denote $u = c^n$ and $v = d^n$. Thus $u$ and $v$ are two compact elements of $A$ such that $u\not\leq p$, $v\not\leq q$ (because $p,q$ are $m$ - prime elements) and $uv = 0$.

\end{proof}
\begin{propozitie}
Assume that $A$ is a coherent quantale. The following properties are equivalent:
\usecounter{nr}
\begin{list}{(\arabic{nr})}{\usecounter{nr}}
\item $A$ is hyperarchimedean;
\item $L(A)$ is a Boolean algebra;
\item For all distinct prime ideals $I, J$ of $L(A)$ there exist two elements $x, y$ of $L(A)$ such that $x\notin I$, $y\notin J$ and $x\land y= 0$;
\item For all distinct $p,q\in Spec(A)$ there exist $c,d\in K(A)$ such that $c\not\leq p$, $d\not\leq q$ and $cd = 0$;

\end{list}
\end{propozitie}

\begin{proof}
$(1)\Leftrightarrow(2)$ By Proposition 6.2.

$(2)\Leftrightarrow(3)$ By Lemma 6.3.

$(3)\Leftrightarrow(1)$ By Proposition 6.4.

\end{proof}

\begin{corolar}
If $A$ is a coherent quantale then it is hyperarchimedean if and only if for all distinct $p,q\in Spec(A)$ there exist $c,d\in K(A)$ such that  $c\not\leq p$, $d\not\leq q$  and $cd = \rho(0)$.
\end{corolar}

\begin{lema}
\cite{Tar1}
Assume that $X$ and $Y$ are topological spaces, $X$ is compact and $Y$ is Hausdorff. Any continuous function $f: X \rightarrow Y$ is a closed map. Moreover, if $f$ is bijective, then it is a homeomorphism.
\end{lema}

\begin{teorema}
If $A$ is a coherent quantale then the following are equivalent:
\usecounter{nr}
\begin{list}{(\arabic{nr})}{\usecounter{nr}}
\item $A$ is hyperarchimedean;
\item For all distinct $p,q\in Spec(A)$ there exist $c,d\in K(A)$ such that  $c\not\leq p$, $d\not\leq q$  and $cd = 0$;
\item $Spec_Z(A)$ is Hausdorff;
\item $Spec_Z(A)$ is a Boolean space;
\item $\mathcal{Z}$ = $\mathcal{P}$;
\item $Spec_F(A)$ is Hausdorff;
\item $Spec_F(A)$ is Boolean space;
\item $\mathcal{Z}$ = $\mathcal{F}$.
\end{list}

\end{teorema}

\begin{proof}

The equivalence $(1)\Leftrightarrow(2)$ follows from Proposition 6.5 and the equivalences $(3)\Leftrightarrow(4)$, $(6)\Leftrightarrow(7)$ are well - known from the general topology.

$(2)\Rightarrow(3)$
Let $p,q$ be two distinct elements of $Spec(A))$. Then there exist $c,d\in K(A)$ such that  $c\not\leq p$, $d\not\leq q$  and $cd = 0$. It follows that $p\in D(c)$, $q\in D(d)$ and $D(c) \bigcap D(d)$ = $D(cd)$ = $\emptyset$. Thus $Spec_Z(A)$ is a Hausdorff space.

$(3)\Rightarrow(5)$ Assume that $Spec_Z(A)$ is a Hausdorff space. Recall from Remark 5.4,(ii) that the identity function $id:Spec_P(A)\rightarrow Spec_Z(A)$ is continuous. Since $Spec_P(A)$ is compact (cf. Lemma 5.1) and  $Spec_Z(A)$ is Hausdorff, by Lemma 6.7 it follows that the map $id:Spec_P(A)\rightarrow Spec_Z(A)$ is a homeomorphism, hence $\mathcal{Z}$ = $\mathcal{P}$.

$(5)\Rightarrow(1)$ Assume that $\mathcal{Z}$ = $\mathcal{P}$ and $p\in Spec(A)$. We want to show that $p\in Max(A)$. According to Remark 5.4(i) and the hypothesis $(5)$ we have $\mathcal{F}\subseteq \mathcal{P}$ = $\mathcal{Z}$. From  $\mathcal{F}\subseteq \mathcal{P}$ it follows that any closed set in $Spec_F(A)$ is closed in $Spec_Z(A)$, so $cl_Z(\{p\})\subseteq cl_P(\{p\})$. Thus by applying Proposition 5.6 one gets $V(p) \subseteq  \Lambda(p)$. Thus $V(p) = \{p\}$, so $p\in Max(A)$. It follows that $Spec(A) = Max(A)$, hence, by Proposition 6.2 we conclude that $A$ is hyperarchimedean.

$(1)\Rightarrow(6)$ Assume that $A$ is hyperarchimedean and $p,q$ are distinct elements of $Spec(A)$, hence, by Proposition 6.2 we have $p,q\in Max(A)$, therefore $p\lor q = 1$. Since $1\in K(A)$ there exist $p,q\in K(A)$ such that $c\leq p$, $d\leq q$ and $c\lor d = 1$. Then $p\in V(c)$, $q\in V(d)$ and $V(c), V(d)$ are open sets of $Spec_F(A)$ such that $V(c)\bigcap V(d)$ = $V(c\lor d)$ = $V(1)$ = $\emptyset$. It follows that $Spec_F(A)$ is a Hausdorff space.

$(6)\Rightarrow(1)$ Assume that $Spec_F(A)$ is a Hausdorff space. Let $p\in Spec(A)$ and $q\in Max(A)$ such that $p\leq q$. Since $Spec_F(A)$ is Hausdorff we have $cl_F(\{q\}) = \{q\}$. According to Proposition 5.6 we have $p\in\Lambda(q)$ = $cl_F(\{q\})$ = $\{q\}$, hence $p = q$. Thus $Spec(A) = Max(A)$, so $A$ is hyperarchimedean.

$(6)\Rightarrow(8)$ Assume that $Spec_F(A)$ is a Hausdorff space. The identity function $id: Spec_P(A) \rightarrow Spec_F(A)$ is continuous, $Spec_F(A)$ is compact and $Spec_F(A)$ is Hausdorff. Hence by Lemma 6.7 it follows that $id: Spec_P(A) \rightarrow Spec_F(A)$ is a homeomorphism, so $\mathcal{F}$ = $\mathcal{P}$. According to the previous proof of the implications $(6)\Rightarrow(1)$, and $(1)\Rightarrow(5)$ we have $\mathcal{Z}$ = $\mathcal{P}$, therefore $\mathcal{Z}$ = $\mathcal{F}$.

$(3)\Rightarrow(5)$ Assume that $\mathcal{Z}$ = $\mathcal{F}$. If $p\in Spec(A)$ then $\Lambda(p)$ = $cl_F(\{p\})$ = $\{p\}$ = $V(p)$, hence $p\in Max(A)$. Thus $Spec(A) = Max(A)$, so the quantale $A$ is hyperarchimedean.
\end{proof}
\begin{remarca}
If we apply the previous theorem to the quantale $Id(R)$ of the ideals of a commutative ring $R$ then we obtain the main part of Theorem 3.3 from \cite{Aghajani}.
\end{remarca}

\section{Flat topology on the maximal spectrum}

 \hspace{0.5cm} In this section we shall study the flat topology on the maximal spectrum of coherent quantales in order to obtain new results on the normal and $B$ - normal quantales.

We fix a coherent quantale $A$. Recall that $Max_F(A)$ is the maximal spectrum $Max(A)$ of $A$ endowed with the restriction of the flat topology of $Spec(A)$.

\begin{lema}
If $c\in K(A)$ then $V(c)\bigcap Max(A)$ is a clopen set of $Max_F(A)$.
\end{lema}

\begin{proof}
 Assume that $c$ is a compact element of $A$. According to Remark 5.3,(ii), $V(c)\bigcap Max(A)$ is an open set of $Max_F(A)$. It remains to prove that the set $D(c)\bigcap Max(A)$ is open in $Max_F(A)$. Let $p\in D(c)\bigcap Max(A)$, hence $c\not\leq p$ and $p\in Max(A)$. If $c\lor p < 1$ then $p < c\lor p\leq q$ for some $q\in Max(A)$, contradicting the maximality of $p$. Thus $c\lor p = 1$, hence there exists $d\in K(A)$ such that $d\leq p$ and $c\lor d = 1$. One obtains $V(c)\bigcap V(d)$ = $V(c\lor d)$ = $V(1)$ = $\emptyset$, hence $V(d)\subset D(c)$. It follows that $p\in V(d)\bigcap Max(A) \subseteq D(c)\bigcap Max(A)$, so $D(c)\bigcap Max(A)$ is an open subset of $Max_F(A)$.
\end{proof}

\begin{propozitie}
The topological space $Max_F(A)$ is Hausdorff and zero - dimensional.
\end{propozitie}

\begin{proof}
Let $p,q$ be two distinct maximal elements of $A$, hence $p\lor q = 1$. Thus there exist $c,d\in K(A)$ such that $c\leq p$, $d\leq q$ and $c\lor d = 1$, therefore $p\in V(c)$, $d \in V(d)$ and $V(c)\bigcap V(d)$ = $V(c\lor d)$ = $V(1)$ = $\emptyset$. It results that $Max_F(A)$ is a Hausdorff space. In accordance to Lemma 7.1, the family $(V(c)\bigcap Max(A))_{c\in K(A)}$ is a basis of clopen sets for $Max_F(A)$, so this topological space is zero - dimensional.
\end{proof}

Following \cite{Cheptea1} we shall denote $r(A)=\bigwedge Max(A)$ . One remarks that $r(A)$ extends the notion of Jacobson radical of a commutative ring. It is obvious that $\rho(0)\leq r(A)$.

\begin{teorema}
$Max_F(A)$ is compact if and only if $[r(A))_A$ is a hyperarchimedean quantale.
\end{teorema}

\begin{proof}

($\Rightarrow$) Assume that $Max_F(A)$ is compact. First we shall prove that $L(A) / (r(A))^\ast$ is a Boolean algebra. Let $c$ be a compact element of $A$ such that $\lambda_A(c) /(r(A))^\ast \not= 0 /(r(A))^\ast$, so $\lambda_A(c) \not\in (r(A))^\ast$. By Lemma 3.5 we have $c \not\leq r(A)$, so there exists $m_c \in Max(A)$ such that $c \not\leq m_c$. For any $m \in Max(A)$ we have $c \leq m$ or $c \not\leq m$; if $c \not\leq m$ then there exists $d_m \in K(A)$ such that $d_m \leq m$ and $c \lor d_m=1$. Since $c \not\leq m_c$, the family $\{m \in Max(A)| c\not\leq m\}$ is non-empty. One remarks that  $Max(A)=\{m \in Max(A)|c \leq m\} \cup \{m \in Max(A)|c\not\leq m\} \subseteq V(c) \cup \displaystyle \bigcup_{c \not\leq m}V(d_m)$.

By hypothesis $Max_F(A)$ is compact so there exist $m_1, \ldots, m_n \in Max(A)$ such that $c \not\leq m_i$ for $i=1, \ldots, n$ and $Max(A) \subseteq V(c) \cup \displaystyle \bigcup_{i = 1}^n V(d_{m_i})$. Let us denote $d_i=d_{m_i}$ for $i=1, \ldots, n$ and $d=d_1 d_2 \ldots d_m$. Therefore $Max(A) \subseteq V(c) \cup V(d)=V(c)$, so $cd \leq r(A)$. By Lemma 3.5, $cd \leq r(A)$ implies $\lambda_A(cd) \in (r(A))^\ast$.

According to Lemma 2.1 (i), from $c \lor d_i=1$, $i=1, \ldots, n$ one gets $c \lor d=1$. Since $\lambda_A(c) \lor \lambda_A(d)=\lambda_A(c \lor d)=
\lambda_A(1)=1$ and $\lambda_A(c) \land \lambda_A(d)=\lambda_A(cd)$, the following equalities hold:

$\lambda_A(c) / (r(A))^\ast \lor \lambda_A(d) /(r(A))^\ast=1 /(r(A))^\ast$;

$\lambda_A(c) / (r(A))^\ast \land \lambda_A(d) /(r(A))^\ast=\lambda_A(cd)/(r(A))^\ast=0/(r(A))^\ast$.

It follows that $L(A) / (r(A))^\ast$ is a Boolean algebra. By Proposition 6 of \cite{Cheptea1}, the lattices $L([r(A))_A)$ and $L(A)/(r(A))^\ast$ are isomorphic, so the reticulation $L([r(A))_A)$ of the quantale $[r(A)_A$ is a Boolean algebra. Applying Proposition 6.1, it follows that $[r(A))_A$ is a hyperarchimedian quantale.

($\Leftarrow$) Assume that the quantale $[r(A))_A$ is hyperarchimedean. By Proposition 6.2 we have $Max_F([r(A))_A)=Spec_F([r(A))_A)$, so $Max_F([r(A))_A)$ is compact (cf. Lemma 5.1). It is easy to see that $Max_F(A)=M_F([r(A))_A)$, so $Max_F(A)$ is compact.

\end{proof}

\begin{propozitie}
The topology of $Max_F(A)$ is finer than the topology of $Max_Z(A)$.
\end{propozitie}
\begin{proof}
A basic open subset of $Max_Z(A)$ has the form $U = Max(A)\bigcap D(c)$, for some $c\in K(A)$. Let us consider an element $m\in U$ so $m\in Max(A)$ and $c\not\leq m$, hence $c\lor m = 1$. Thus we have $c\lor d = 1$ for some $d\in K(A)$ with $d\leq m$. Therefore $V(c)\bigcap V(d)$ = $V(c\lor d)$ = $V(1) = \emptyset$, so $m\in Max(A)\bigcap V(d)$ and
 $Max(A)\bigcap V(d)$ is included in $U$. It follows that $U$ is an open subset of $Max_F(A)$.
\end{proof}

The following proposition characterizes the quantales $A$ for which $Max_Z(A)$ and $Max_F(A)$ coincide.

\begin{propozitie}
If $A$ is a coherent quantale then the following are equivalent:
\usecounter{nr}
\begin{list}{(\arabic{nr})}{\usecounter{nr}}
\item $Max_F(A)$ is compact;

\item The topological spaces $Max_Z(A)$ and $Max_F(A)$ coincide;

\item $[r(A))_A$ is a hyperarchimedean quantale.

\end{list}

\end{propozitie}
\begin{proof}
The equivalence of $(1)$ and $(3)$ follows by Theorem 7.3. We remark that the following equalities hold: $Max_F(A)$ =  $Max_F([r(A))_A)$ and  $Max_Z([r(A))_A)$ = $Max_Z(A)$. According to Theorem 6.8 the assertions $(2)$ and $(3)$ are equivalent.
\end{proof}

Following \cite{Johnstone}, p.199 we say that a commutative ring $R$ is a Gelfand ring if each prime ideal of $R$ is contained in a unique maximal ideal. Recall from \cite{Simmons},\cite{Johnstone} that a bounded distributive lattice $L$ is called normal if for all elements $x,y\in L$ such that $x\lor y$ = $1$ there exist $u,v\in L$ such that $x\lor u = y\lor v = 1$ and $uv = 0$. We know from \cite{Johnstone}, p.68 that a bounded distributive lattice $L$ is normal if and only if each prime ideal of $L$ is contained in a unique maximal ideal. The normal quantales  were introduced in \cite{PasekaRN} as an abstractization of the lattices of ideals of Gelfand rings and normal lattices.

According to \cite{PasekaRN}, a quantale $A$ is said to be normal if for all $a,b\in A$ such that $a\lor b = 1$ there exist $e,f\in A$ such that $a\lor e$ = $b\lor f$ = $1$ and $ef = 0$. If $1\in K(A)$ then $A$ is normal if and only if for all $c,d\in K(A)$ such that $c\lor d = 1$ there exist $e,f\in K(A)$ such that $c\lor e$ = $d\lor f$ = $1$ and $ef = 0$ (cf. Lemma 20 of \cite{Cheptea1}). One observes that a commutative ring $R$ is a Gelfand ring iff $Id(R)$ is a normal quantale and a bounded distributive lattice $L$ is normal iff $Id(L)$ is a normal frame.

The normal quantales offer an abstract framework in order to unify some algebraic and  topological properties of commutative Gelfand rings \cite{Johnstone}, \cite{b}, \cite{c}, \cite{lu}, \cite{se}, normal lattices \cite{Johnstone}, \cite{GeorgescuVoiculescu}, \cite{Pawar}, \cite{Simmons}, commutative unital $l$ - groups \cite{Birkhoff}, $F$ - rings \cite{Birkhoff}, \cite{Johnstone}, $MV$ - algebras and $BL$ - algebras \cite{Galatos}, \cite{g}, Gelfand residuated lattices \cite{GCM}, etc.

Let us fix a coherent quantale $A$.
\begin{propozitie}
\cite{Cheptea1}
The quantale $A$ is normal if and only if the reticulation $L(A)$ is a normal lattice ( in the sense of \cite{Simmons},\cite{Johnstone}).
\end{propozitie}

\begin{propozitie}
\cite{PasekaRN},\cite{GeorgescuVoiculescu2},\cite{SimmonsC}
If $A$ is a coherent quantale then the following are equivalent:
\usecounter{nr}
\begin{list}{(\arabic{nr})}{\usecounter{nr}}
\item $A$ is a normal quantale;
\item For all distinct $m,n\in Max(A)$ there exist $c_1,c_2\in K(A)$ such that  $c_1\not\leq m$, $c_2\not\leq n$  and $c_1c_2 = 0$;
\item The inclusion $Max(A)\subseteq Spec(A)$ is a Hausdorff embedding (i.e. any distinct points in $Max(A)$ have disjoint neighbourhoods in $Spec_Z(A))$;
\item For any $p\in Spec(A)$ there exists a unique $m\in Max(A)$ such that $p\leq m$;
\item $Spec_Z(A)$ is a normal space;
\item The inclusion $Max_Z(A)\subseteq Spec_Z(A)$ has a continuous retraction $\gamma:Spec_Z(A)\rightarrow Max_Z(A)$;
\item If $m\in Max(A)$ then $\Lambda(m)$ is a closed subset of $Spec_Z(A)$.
\end{list}

\end{propozitie}

\begin{remarca}
A proof of the previous proposition can be obtained by using Proposition 7.6 and some characterizations of normal lattices given in \cite{GeorgescuVoiculescu}, \cite{Johnstone}, \cite{Pawar}, \cite{Simmons}.

\end{remarca}

\begin{teorema}
Assume that $A$ is a normal quantale. Then the retraction map $\gamma:Spec(A)\rightarrow Max(A)$ is flat continuous if and only if $Max_F(A)$ is a compact space.
\end{teorema}

\begin{proof}
Assume that the retraction map $\gamma:Spec_F(A)\rightarrow Max_F(A)$ is continuous. Since $Spec_F(A)$ is compact it follows that $Max_F(A)$ is also compact.

Conversely, assume that $Max_F(A)$ is compact, hence by Proposition 7.2 it is a Boolean space. By Proposition 7.3 it results that $[(r(A)_A)$ is a hyperarchimedean quantale. Applying the condition (4) of Theorem 6.7 one gets $Max_Z([(r(A)_A)$ = $Max_F([(r(A))_A)$. We remark that $Max_Z(A)$ and $Max_F([(r(A))_A)$ are homeomorphic, thus  by Theorem 6.7(4) it follows that $Max_Z(A)$ is a Boolean space. Thus $(Max(A)\bigcap D(c))_{c\in K(A)}$ is a basis of clopen sets in $Max_Z(A)$.

Let us consider an element $c\in K(A)$; in accordance to the continuity of $\gamma:Spec_F(A)\rightarrow Max_F(A)$, it follows that $\gamma^{-1} (Max(A)\bigcap D(c))$ is a clopen set in $Spec_Z(A)$. Applying Lemma 2.4 of \cite{Cheptea1} we find an element $e\in B(A)$ such that  $\gamma^{-1} (Max(A)\bigcap D(c))$ = $V(e)$. Therefore $\gamma^{-1} (Max(A)\bigcap D(c))$ is a clopen subset of $Max_F(A)$ (cf. Lemma 7.1), hence the map $\gamma:Spec_F(A)\rightarrow Max_F(A)$ is continuous.

\end{proof}

\begin{propozitie}

If $Max_Z(A)$ is Hausdorff and $\rho(0) = r(A)$ then the quantale $A$ is normal.

\end{propozitie}
\begin{proof}
Assume by absurdum that the quantale $A$ is not normal, so by Proposition 7.7,(4) there exist $p\in Spec(A)$ and $q,r\in Max(A)$ such that $q\neq r$, $p\leq q$ and $p\leq r$. Since $Max_Z(A)$ is Hausdorff there exist $c,d\in K(A)$ such that $q\in D(c)$, $r\in D(d)$ and $D(cd)\bigcap Max(A)$ =
$D(c)\bigcap D(d)\bigcap Max(A)$ = $\emptyset$.

If $cd\not\leq \rho(0)$ then $cd\not\leq r(A)$, so  $cd\not\leq m$ for some $m\in Max(A)$. It results that $m\in D(cd)\bigcap Max(A)$, contradicting $D(cd)\bigcap Max(A)$ = $\emptyset$. Thus $cd\leq \rho(0)$, hence one gets $c\leq p$ or $d\leq p$. If $c\leq p$ then $c\leq q$, contradicting $q\in D(c)$; similarly, $c\leq p$
contradicts $r\in D(d)$. We conclude that $A$ is normal.
\end{proof}

\begin{corolar}
$Max_Z(A)$ is a Hausdorff space if and only if $[r(a))_A$ is a normal quantale.
\end{corolar}

\begin{proof}
We observe that the quantale $C$ = $[r(a))_A$ verifies the conditions $\rho_C(0)$ = $r(C)$ and $Max_Z(A)$ = $Max_C(A)$. Applying Proposition 7.10 to the quantale $C$ the following equivalences hold: $Max_Z(A)$ is Hausdorff iff $Max_Z(C)$ is Hausdorff iff $C$ is a normal quantale.
\end{proof}

Following \cite{Cheptea1} we say that a quantale $A$ is said to be $B$ - normal if for all $c,d\in K(A)$ there exist $e,f\in B(A)$ such that $c\lor e$ = $d\lor f$ = 1 and $cd = 0$. If the B - normal quantale $A$ is a frame then we shall say that $A$ is a $B$ - normal frame. The $B$ - normal quantales constitute an abstract setting in which we can generalize various results on the clean commutative rings \cite{b}, \cite{a}, the $B$ - normal (bounded distributive ) lattices \cite{Cignoli}, the clean unital $l$ - groups \cite{f}, the quasi - local $BL$ - algebras \cite{g}, the quasi - local residuated lattices \cite{Muresan},etc.

\begin{lema}
If $A$ is normal quantale then $(D(e)\bigcap Max(A))_{e\in B(A)}$ is the family of the clopen subsets of $Max_Z(A)$.
\end{lema}
\begin{proof}
Let $K$ be a clopen subset of $Max_Z(A)$. If $\gamma: Spec_Z(A)\rightarrow Max_Z(A)$ is the continuous retract of the inclusion $Max_Z(A)\subseteq Spec_Z(A)$, then  $L = \gamma^{-1}(K)$ is a clopen set in $Spec_Z(A)$. By Lemma 24 of \cite{Cheptea1} there exists an element $e\in B(A)$ such that $L = D(e)$. Thus $K$ = $\gamma(D(e))$ = $\{\gamma(p)| p\in Spec(A),  e\not\leq p \}$. It is easy to see that for all $p\in Spec(A)$ we have $e\leq p$ iff $e\leq \gamma(p)$, hence $K$ =  $\{\gamma(p)| p\in Spec(A), e\not\leq \gamma(p)\}$ = $D(e)\bigcap Max(A)$.
\end{proof}

The following theorem contains some conditions that characterize the $B$ - normal algebras.

\begin{teorema}
If $A$ is a coherent quantale then the following are equivalent:
\usecounter{nr}
\begin{list}{(\arabic{nr})}{\usecounter{nr}}
\item $A$ is $B$ - normal;
\item $R(A)$ is a $B$ - normal frame;
\item The reticulation $L(A)$ is a $B$ - normal lattice;
\item For all distinct $p,q\in Max(A)$ there exists $e\in B(A)$ such that $e\leq p$ and $\neg e \leq q$;
\item $A$ is a normal quantale and $Max_Z(A)$ is a zero - dimensional space;
\item $A$ is a normal quantale and $Max_Z(A)$ is a Boolean space;
\item The family  $(D(e)\bigcap Max(A))_{e\in B(A)}$ is a basis of open sets for $Max_Z(A)$;
\item The function $s_A|_{Max(A)}: Max_Z(A)\rightarrow Sp(A)$ is a homeomorphism.
\end{list}

\end{teorema}

\begin{proof}
The equivalence of the properties $(1), (2), (3), (5)$ and $(6)$ was established in \cite{Cheptea1}, hence it remains to prove the equivalence of the other conditions.

$(1)\Rightarrow(4)$
Let $p,q$ be two distinct maximal elements of $A$, hence $p\lor q = 1$. Since $A$ is $B$ - normal, there exist $e,f\in B(A)$   such that $p\lor f$  = $q\lor e$ = $1$ and $ef = 0$. From $p\lor f$  =  $q\lor e$  = $1$  one gets $f\not\leq p$, $e\not\leq q$, hence $\neg f \leq p$ and  $\neg e\leq q$.  The equality $ef = 0$ implies $e\leq \neg f \leq p$.

$(6)\Rightarrow(7)$
Since $Max_Z(A)$ is a Boolean space, the family of its clopen subsets is a basis of open sets. By Lemma 7.12, the family $(D(e)\bigcap Max(A))_{e\in B(A)}$ is exactly this basis of open sets for $Max_Z(A)$.

$(7)\Rightarrow(4)$
Let $p,q$ be two distinct maximal elements of $A$. We observe that $U$ = $Spec(A) - \{q\}$ = $Spec(A) - V(q)$ = $D(q)$ is open in $Spec_Z(A)$, so $U\bigcap Max(A)$ is an open subset of $Max_Z(A)$ that contains $p$. In accordance to the hypothesis, there exists $e\in B(A)$ such that $p \in D(e)\bigcap Max(A)\subseteq U\bigcap Max(A)$. It follows that $e\not\leq p$ and $e\leq q$, so $\neg e\leq p$ and $e\leq q$.

$(1)\Rightarrow(8)$ Assume that $A$ is $B$ - normal. In accordance to Proposition 5.9,  $s_A|_{Max(A)}: Spec_Z(A)\rightarrow Sp(A)$ is a surjective continuous map. We shall prove that the restriction of $s_A$ to $Max(A)$ is injective. Let $m,n\in Max(A)$ such that $m\neq n$. We know that the conditions $(1)$ and $(4)$ are equivalent, so there exists $e\in B(A)$ such that $e\leq m$, $\neg e \leq n$, hence $e\leq s_A(m)$ and $e\not\leq s_A(n)$. It follows that $s_A(m)\neq s_A(n)$, so $s_A$ is injective.

In order to prove that $s_A|_{Max(A)}: Max_Z(A)\rightarrow Sp(A)$ is surjective assume that $q\in Sp(A)$, hence $q = s_A(p)$, for some $p\in Spec(A)$. Let $\gamma(p)$ be the unique maximal element of $A$ such that $p\leq \gamma(p)$. Thus $q$ = $s_A(p)\leq s_A(\gamma(p))$, so $q = \gamma(p)$, because $q$ and $\gamma(p)$ are max- regular elements. We know already that $(1)$ and $(6)$ are equivalent, so $Max(A)$ is a Boolean space. By Proposition 5.9, $Sp(A)$ is also a Boolean space. Therefore, by applying Lemma 6.8 it follows that $s_A|_{Max(A)}: Max_Z(A)\rightarrow Sp(A)$ is a homeomorphism.

$(8)\Rightarrow(6)$
Taking into account the hypothesis $(8)$, it follows that the function $(s_A|_{Max(A)})^{-1}\circ s_A:Spec_Z(A)\rightarrow Max_Z(A)$ is a continuous retraction of the inclusion $Max_Z(A)\subset Spec_Z(A)$, so $A$ is a normal quantale. Moreover, by $(8)$ and Proposition 5.9 it follows that $Max_Z(A)$ is a Boolean space.

\end{proof}

\begin{corolar}
Let $A$ be a normal quantale. If $[r(A))_A$ is a hyperarchimedean quantale then $A$ is $B$ - normal.
\end{corolar}
\begin{proof}
Assume that $[r(A))_A$ is hyperarchimedean, so by Proposition 4.5 we have $Max_F(A)$ = $Max_Z(A)$. Therefore by using Proposition 7.2 it follows that $Max_Z(A)$ is zero - dimensional. Applying Theorem 7.13,(5) one gets that $A$ is a $B$ - normal quantale.
\end{proof}

\section{Flat topology on the minimal prime spectrum}

If $A$ is a quantale then we denote by $Min(A)$ the set of minimal $m$ - prime elements of $A$; $Min(A)$ is called the minimal prime spectrum of $A$. If $1\in K(A)$ then for any $p\in Spec(A)$ there exists $q\in Min(A)$ such that $q\leq p$. For any bounded distributive lattice $L$ we denote by $Min_{Id}(L)$ the set of minimal prime ideals in $L$; $Min_{Id}(L)$ is the minimal prime spectrum of the frame $Id(L)$.

We will obtain a description of the minimal $m$ - prime elements of a coherent quantale $A$ by using the reticulation. First we remember from \cite{Simmons} the following result.

\begin{lema}
A prime ideal $P$ of a bounded distributive lattice $L$ is minimal prime if and only if for all $x\in P$ we have $Ann(x)\not\subseteq P$.
\end{lema}
Let us fix a coherent quantale $A$.

\begin{lema}
If $c\in K(A)$ and $p\in Spec(A)$ then $Ann(\lambda_A(c))\subseteq p^{\ast}$ if and only if $c\rightarrow \rho(0)\leq p$.
\end{lema}

\begin{proof}
If $Ann(\lambda_A(c))\subseteq p^{\ast}$, then by using Lemma 3.4 and Proposition 4.5, one gets $c\rightarrow \rho(0) \leq \rho(c\rightarrow \rho(0))$ = $((c\rightarrow \rho(0))^{\ast})_{\ast}$ = $(Ann(\lambda_A(c)))_{\ast}\leq (p^{\ast})_{\ast}$ = $p$. Conversely, if $c\rightarrow \rho(0)\leq p$, then by using Proposition 4.5 we have $Ann(\lambda_A(c)) = (c\rightarrow \rho(0))^{\ast}\subseteq p^{\ast}$.
\end{proof}

\begin{propozitie}
If $p\in Spec(A)$ then the following are equivalent:
\usecounter{nr}
\begin{list}{(\arabic{nr})}{\usecounter{nr}}

\item $p\in Min(A)$;

\item $p^{\ast}\in Min_{Id}(L(A))$;

\item For all $c\in K(A))$, $\lambda_A(p)\in p^{\ast}$ implies $Ann(\lambda_A(p))\not\subseteq p^{\ast}$;

\item For all $c\in K(A))$, $c\leq p$ if and only if $c\rightarrow \rho(0)\not\leq p$.
\end{list}
\end{propozitie}

\begin{proof}
 $(1)\Leftrightarrow (2)$ Let us consider the order - preserving map $u: Spec(A)\rightarrow Spec_{Id}(L(A))$ defined by $u(p)$ = $p^{\ast}$, for all $p\in Spec(A)$. According to Proposition 3.7, $u$ is an order - isomorphism, hence the conditions $(1)$ and $(2)$ are equivalent.

$(2)\Leftrightarrow (3)$ By Lemma 3.3,(3), $p^{\ast}$ is a prime ideal of the lattice $L(A)$. Therefore, by using Lemma 8.1 it follows that the properties $(2)$ and $(3)$ are equivalent.

$(3)\Leftrightarrow (4)$ By Lemmas 3.6 and 8.2.

\end{proof}

\begin{corolar}

If $A$ is semiprime and $p\in Spec(A)$ then $p\in Min(A)$ if and only if for all $c\in K(A)$, $c\leq p$ implies $c^{\perp}\not\leq p$.

\end{corolar}

Let us denote by $Min_Z(A)$ (resp. $Min_F(A)$) the topological space obtained by restricting the topology of $Spec_Z(A)$ (resp. $Spec_F(A)$) to $Min(A)$. Similarly, for a bounded distributive lattice $L$ we denote by $Min_{Id, Z}(L)$ (resp. $Min_{Id,F}(L)$) the space obtained by restricting to $Min_{Id}(L)$ the Stone topology (resp. the flat topology) of $Spec_{Id}(L)$.

\begin{lema}
The topological spaces $Min_Z(A)$ and $Min_{Id, Z}(L)$ (resp. $Min_F(A)$ and $Min_{Id,F}(L)$)) are homeomorphic.
\end{lema}

\begin{corolar}
$Min_Z(A)$ is a zero - dimensional Hausdorff space and $Min_F(A)$ is a compact $T1$  space.
\end{corolar}

\begin{proof}
By \cite{Speed}, $Min_{Id, Z}(L)$ is a zero - dimensional Hausdorff space and  $Min_{Id, F}(L)$ is a compact $T1$  space.
\end{proof}

\begin{propozitie}
\cite{Speed}
If $L$ is a bounded distributive lattice then the following are equivalent:
\usecounter{nr}
\begin{list}{(\arabic{nr})}{\usecounter{nr}}

\item $Min_{Id, Z}(L)$ = $Min_{Id, F}(L)$;

\item $Min_{Id, Z}(L)$ is a compact space;

\item $Min_{Id, Z}(L)$ is a Boolean space;

\item For any $s\in L$ there exists $y\in L$ such that $x\land y = 0$ and $Ann(x\lor y)$ = $\{0\}$.

\end{list}
\end{propozitie}

\begin{teorema}
If $A$ is a semiprime quantale then the following are equivalent:
\usecounter{nr}
\begin{list}{(\arabic{nr})}{\usecounter{nr}}

\item $Min_Z(A)$ = $Min_F(A)$;

\item $Min_Z(A)$ is a compact space;

\item $Min_Z(A)$ is a Boolean space;

\item For any $c\in K(A)$ there exists $d\in K(A)$ such that $cd = 0$ and $(c\lor d)^{\perp}$ = $0$.

\end{list}

\end{teorema}

\begin{proof}

$(1)\Leftrightarrow (2)\Leftrightarrow (3)$
These equivalences follow by applying Lemma 8.5 and taking in account the equivalence of the assertions (i), (ii) and (iii) from Proposition 8.7.

$(1)\Rightarrow (2)$
Assume that $c\in K(A)$. By Lemma 8.5 we have $Min_{Id, Z}(L(A))$ = $Min_{Id, F}(L(A))$, therefore by applying Proposition 8.7 to the lattice $L(A)$ there exists $d\in K(A)$ such that $\lambda_A(cd)$ = $\lambda_A(c)\land \lambda_A(d) = 0$ and $Ann(\lambda_A(c\lor d))$ = $Ann(\lambda_A(c)\lor \lambda_A(d))$ = $\{0\}$. The quantale $A$ is semiprime, hence by using Lemma 3.2,(9) and Proposition 4.5, one obtains $cd = 0$ and $((c\lor d)^{\perp})^{\ast}$ = $\{0\}$. If $z\in K(A)$ and $z\leq (c\lor d)^{\perp}$ then $\lambda_A(z)\in ((c\lor d)^{\perp})^{\ast}$, so $\lambda_A(z) = 0$. Since $A$ is semiprime it follows that $z =0$ (cf. Lemma 3.2,(9)). We conclude that $(c\lor d)^{\perp}$ = $0$.

$(4)\Rightarrow (1)$
Assume that $x\in L(A)$ hence $x = \lambda_A(c)$ for some $c\in K(A)$. Then there exists $d\in K(A)$ such that $cd = 0$ and $(c\lor d)^{\perp}$ = $0$. Denoting $y = \lambda_A(d)$ we obtain $x\land y$ = $\lambda_A(cd)$ = $0$ and $Ann(x\lor y)$ = $Ann(\lambda_A(c\lor d))$ = $((c\lor d)^{\perp})^{\ast}$ = $0$. By Proposition 8.7 we have $Min_{Id, Z}(L(A))$ = $Min_{Id, F}(L(A))$, hence $Min_Z(A)$ = $Min_F(A)$.

\end{proof}

Recall from \cite{Aghajani} that an $mp$ - ring is a commutative ring $R$ with the property that each prime ideal of $R$ contains a unique minimal prime ideal. Let us extend this notion to quantales: a quantale $A$ is an $mp$ - quantale if for any $p\in Spec(A)$ there exist a unique $q\in Min(A)$ such that $q\leq p$. An $mp$ - frame is an $mp$ - quantale wich is a frame. We remark that a ring $R$ is an $mp$ - ring if and only if the quantale $Id(R)$ of ideals of $R$ is an $mp$ - quantale.

The $mp$ - quantales can be related to the conormal lattices, introduced by Cornish in \cite{Cornish} under the name of "normal lattices". According to \cite{Simmons},\cite{Johnstone}, a conormal lattice is a bounded distributive lattice $L$ such that for all $x,y\in L$ with $x\land y = 0$ there exist $u,v\in L$ having the properties $x\land u$ = $y\land v$ = $0$ and $u\lor v$ = $1$. In \cite{Cornish} Cornish obtained several characterizations of the conormal lattices.

\begin{propozitie}
\cite{Cornish}
A bounded distributive lattice $L$ is conormal if and only if any prime ideal of $L$ contains a unique minimal prime ideal.
\end{propozitie}

\begin{corolar}
A coherent quantale $A$ is an $mp$ - quantale if and only if the reticulation $L(A)$ is a conormal lattice.
\end{corolar}

\begin{proof}
Recall that the two functions $u_A: Spec(A)\rightarrow Spec_{Id}(L(A))$ and $v_A: Spec_{Id}(L(A))\rightarrow Spec(A)$  from Proposition 3.7 are order - isomorphisms (the order is the inclusion). Thus the corollary follows by using Proposition 8.9.
\end{proof}

By \cite{Hochster} for any bounded distributive lattice $L$ there exists a commutative ring $R$ such that the lattices $L$ and $L(A)$ are isomorphic (see also the discussion from Section 3.13 of \cite{Johnstone}). Thus for any coherent quantale $A$ there exists a commutative ring R such that the lattices $L(A)$ and $L(R)$ are isomorphic (we shall identify these isomorphic lattices). Let us fix this ring $R$ associated with the quantale $A$.

In accordance with Proposition 3.7 we have the following homeomorphisms:

$Spec_Z(A) \xrightarrow[]{u_A} Spec_{Id, Z}(L(A)) \xrightarrow[]{v_R} Spec_Z(R)$ (i)

$Spec_Z(R) \xrightarrow[]{u_R} Spec_{Id, Z}(L(A)) \xrightarrow[]{v_A} Spec_Z(A)$ (ii)

By restricting these four maps to minimal prime spectra one gets the following homeomorphisms:

$Min_Z(A) \xrightarrow[]{u_A} Min_{Id, Z}(L(A)) \xrightarrow[]{v_R} Min_Z(R)$ (iii)

$Min_Z(R) \xrightarrow[]{u_R} Min_{Id, Z}(L(A)) \xrightarrow[]{v_A} Spec_Z(A)$ (iv)

(we denote the restrictions by the same symbols).

\begin{remarca}
Taking into account Proposition 5.5 it is easy to prove that the following maps: $Spec_F(A) \xrightarrow[]{u_A} Spec_{Id, F}(L(A))$, $Spec_{Id, F}(L(A)) \xrightarrow[]{v_A} Spec_F(A)$, $Spec_F(R) \xrightarrow[]{u_R} Spec_{Id, F}(L(A))$ and $Spec_{Id, F}(L(A)) \xrightarrow[]{v_R} Spec_F(R)$ are homeomorphisms.

\end{remarca}

\begin{lema}
The coherent quantale $A$ is an $mp$ - quantale if and only if $R$ is an $mp$ - ring.
\end{lema}

\begin{proof}
We apply Proposition 8.9 to the isomorphic reticulations $L(A)$ and $L(R)$ of the quantale $A$ and the ring $R$.
\end{proof}

\begin{propozitie}
For a coherent quantale $A$ the following are equivalent:
\usecounter{nr}
\begin{list}{(\arabic{nr})}{\usecounter{nr}}

\item The inclusion $Min_F(A)\subseteq Spec_F(A)$ has a flat continuous retraction;

\item The inclusion $Min_F(R)\subseteq Spec_F(R)$ has a flat continuous retraction.

\end{list}
\end{propozitie}

\begin{proof}
According to Remark 8.11, each of these two conditions is equivalent to the following property: the inclusion $Min_{Id,F}(L(A))\subseteq Spec_{Id,F}(L(A))$ has a flat continuous retraction.
\end{proof}

\begin{teorema}
If $A$ is a coherent quantale then the following are equivalent:
\usecounter{nr}
\begin{list}{(\arabic{nr})}{\usecounter{nr}}

\item $A$ is an $mp$ - quantale;

\item For any distinct elements $p,q\in Min(A)$ we have $p\lor q = 1$;

\item $R(A)$ is an $mp$ - frame;

\item $[\rho(0))_A$ is an $mp$ - quantale;

\item The inclusion $Min_F(A)\subseteq Spec_F(A)$ has a flat continuous retraction;

\item $Spec_F(A)$ is a normal space;

\item If $p\in Min(A)$ then $V(p)$ is a closed subset of $Spec_F(A)$.

\end{list}

\end{teorema}

\begin{proof}

$(1)\Rightarrow (2)$
Suppose that $p,q$ are distinct elements of $Min(A)$. If $p\lor q < m$ then $p\lor q\leq m$ for some $m\in Max(A)$. Then there exist two distinct $p,q\in Min(A)$ such that $p\leq m$ and $q\leq m$, contradicting that $A$ is an $mp$ - quantale. It follows that $p\lor q = 1$.

$(2)\Rightarrow (1)$
Assume by absurdum that there exist $p\in Spec(A)$ and two distinct $q,r\in Min(A)$ such that $q\leq p$ and $r\leq p$. Thus $q\lor r = 1$, hence $p = 1$, contradicting that $p\in Spec(A)$.

$(1)\Leftrightarrow (3)$
By Lemma 6 of \cite{Cheptea1} we have $Spec(A) = Spec(R(A))$, hence $Min(A) = Min(R(A))$, so the equivalence of $(1)$ and $(3)$ is immediate.

$(1)\Leftrightarrow (4)$
This equivalence follows from $Spec(A) = Spec([\rho(0))_A)$ and $Min(A) = Min([\rho(0))_A)$.

$(1)\Leftrightarrow (5)$
By using Lemma 8.12, Proposition 8.13 and the equivalence of the conditions (i), (v) from Theorem 6.2 of \cite{Aghajani} it results that the properties $(1)$ and $(5)$ are equivalent.

$(1)\Leftrightarrow (6)$
By Remark 8.11, Lemma 8.12 and Theorem 6.2 of \cite{Aghajani}, it follows that $A$ is an $mp$ - quantale iff $R$ is an $mp$ - ring iff $Spec_F(R)$ is normal iff $Spec_F(A)$ is normal.

$(2)\Rightarrow (7)$
Assume that $p\in Min(A)$ and $q\in D(p)$. Consider an element $r\in Min(A)$ such that $r\leq q$. Since $p\not\leq q$ we have $p\neq r$, hence $p\lor r = 1$  by the hypothesis $(2)$, so there exist $c,d\in K(A)$ such that $c\leq p$, $d\leq r$ and $c\lor d = 1$. Thus $V(c)\bigcap V(d)$ = $V(c\lor d)$ = $V(1)$ = $\emptyset$, so $V(d)\subseteq D(c)\subseteq D(p)$. From $d\leq r\leq q$ one gets $q\in V(d)$. Since $q\in V(d)\subseteq  D(p)$ and $V(d)$ is a basic open set of $Spec_F(A)$ it follows that $D(p)$ is open in $Spec_F(A)$. We conclude that $V(p)$ is closed in $Spec_F(A)$.

$(7)\Rightarrow (2)$
Assume by absurdum that there exist two distinct $p,q\in Min(A)$ such that $p\lor q < 1$, so $p\lor q \leq m$ for some $m\in Max(A)$. Therefore $m\in V(p)$ and $m\in V(q)$, hence $V(p)\bigcap V(q)\neq \emptyset$. Since $V(q)$ is an open neighborhood of $p$ in $Spec_F(A)$ and $V(p)$ is flat closed, one gets $q\in V(p)$. Thus $q\leq p$ so $q = p$ because $p$ and $q$ are minimal $m$ - prime elements. This contradiction shows that for all distinct minimal $m$ - prime elements $p,q$ we have $p\lor q = 1$.

\end{proof}

\begin{remarca}
If $R$ is a commutative ring and $A$ is the quantale $Id(R)$ of ideals of $R$, then applying the previous result one obtains some of the characterizations of $mp$ - rings, contained in Theorem 6.2 of \cite{Aghajani}.
\end{remarca}

Recall from \cite{Al-Ezeh2} that a commutative ring $R$ is said to be an $PF$ - ring if the annihilator of each element of $R$ is a pure ideal. We shall generalize this notion to quantales. Then a quantale $A$ is a $PF$ - quantale if for each $c\in K(A)$, $c^{\perp}$ is a pure element. For any commutative ring $R$, $Id(R)$ is a $PF$ - quantale if and only if $R$ is a $PF$ - ring.

\begin{propozitie}
\cite{Al-Ezeh1}
If $A$ is a bounded distributive lattice $L$ then the following are equivalent
\usecounter{nr}
\begin{list}{(\arabic{nr})}{\usecounter{nr}}
\item L is conormal;

\item For all $x\in L$, $Ann(x)$ is a $\sigma$ - ideal;

\item Any minimal prime ideal of $L$ is a $\sigma$ - ideal.

\end{list}

\end{propozitie}

In other words, a bounded distributive lattice $L$ is conormal if and only if $Id(L)$ is a $PF$ - frame.

In what follows we shall establish a relationship between $PF$ - quantales and $mp$ - quantales. We fix a coherent quantale $A$.

\begin{lema}
Any $PF$ - quantale $A$ is semiprime.
\end{lema}

\begin{proof}
Let $c$ be a compact element of $A$ such that $c^n = 0$ for some integer $n\geq 1$. Then $c^{n-1}\leq (c^{n-1})^{\perp}$, hence $(c^{n-1})^{\perp}$ =   $(c^{n-1})^{\perp}\lor  (c^{n-1})^{\perp} = 1$, because $(c^{n-1})^{\perp}$ is pure. Thus $c^{n-1} \leq (c^{n-1})^\perp=0$ , so $c^{n-1} = 0$. By using many times this argument one gets $c = 0$. According to Lemma 2.4,(2) it follows that $A$ is semiprime.

\end{proof}

\begin{propozitie}
If $A$ is a $PF$ - quantale then the reticulation $L(A)$ is a conormal lattice.
\end{propozitie}

\begin{proof}
By Lemma 8.17, the $PF$ - quantale $A$ is semiprime. Assume that $x\in L(A)$ so $x = \lambda_A(c)$, for some $c\in K(A)$. Applying Proposition 4.5 one obtains $Ann(x)$ = $Ann(\lambda_A(c))$ = $Ann(c^{\ast})$ = $(c^{\perp})^{\ast}$. By hypothesis, $c^{\perp}$ is a pure element of $A$, therefore by Lemma 4.7 it follows that $Ann(x)$ = $(c^{\perp})^{\ast}$ is a $\sigma$ - ideal of the lattice $L(A)$. Applying Proposition 8.16 it follows that $L(A)$ is a conormal lattice.
\end{proof}

\begin{teorema}
For a coherent quantale $A$ the following are equivalent:
\usecounter{nr}
\begin{list}{(\arabic{nr})}{\usecounter{nr}}

\item $A$ is a $PF$ - quantale;

\item $A$ is a semiprime $mp$ - quantale.

\end{list}
\end{teorema}

\begin{proof}

$(1)\Rightarrow(2)$ By Proposition 8.18 and Corollary 8.10.

$(2)\Rightarrow(1)$ In order to prove that $A$ is a $PF$-quantale let us assume that $c\in K(A)$. We shall prove that $c^{\perp}$ is a pure element of $A$. Let $d$ be a compact element of $A$ such that $d\leq c^{\perp}$, hence $\lambda_A(c)\land \lambda_A(d)$ = $\lambda_A(cd)$ = $\lambda_A(0) = 0$, i.e. $\lambda_A(d)\in Ann(c^{\ast})$ = $Ann(\lambda_A(c))$. By Corollary 8.10 $L(A)$ is a conormal lattice, hence $Ann(\lambda_A(c))$ is $\sigma$ - ideal of $L(A)$ (cf. Proposition 8.16). It follows that $Ann(c^{\ast})\lor Ann(d^{\ast})$ = $Ann(\lambda_A(c))\lor Ann(\lambda_A(d))$ = $L(A)$.

According to Lemma 3.4, Proposition 4.5 and Corollary 3.10, the following equalities hold:

$\rho(\rho(c^{\perp})\lor \rho(d^{\perp}))$ = $\rho(((c^{\perp})^{\ast})_{\ast}\lor ((c^{\perp})^{\ast})_{\ast})$ = $\rho((Ann(c^{\ast}))_{\ast}\lor (Ann(d^{\ast}))_{\ast})$ = $(Ann(\lambda_A(c))\lor Ann(\lambda_A(d)))_{\ast}$ = $(L(A))_{\ast}$ = $1$.

By Lemma 2.2,(3) and (6) one gets $c^{\perp}\lor d^{\perp} = 1$, hence $c^{\perp}$ is pure.
\end{proof}

\begin{teorema} For a coherent quantale $A$ consider the following conditions:
\usecounter{nr}
\begin{list}{(\arabic{nr})}{\usecounter{nr}}
\item Any minimal $m$ - prime element of $A$ is pure;

\item $A$ is an $mp$ - quantale.
\end{list}
Then $(1)$ implies $(2)$. If the quantale $A$ is semiprime then the converse implication holds.
\end{teorema}

\begin{proof} First we shall prove that $(1)$ implies $(2)$. According to Corollary 8.10 it suffices to check that the reticulation $L(A)$ is a conormal lattice. Let $P$ be a minimal prime ideal of $L(A)$, hence $P = p^{\ast}$ for some $p\in Min(A)$. By taking into account the hypothesis, it results that $p$ is a pure element of $A$. In accordance to Lemma 4.7, $P = p^{\ast}$ is a $\sigma$ - ideal of $L(A)$, so any minimal prime ideal of $L(A)$ is a $\sigma$ - ideal. Applying Proposition 8.16 it follows that the lattice $L(A)$ is conormal.

Assume now that $A$ is a semiprime $mp$ - quantale and $p\in Min(A)$, so $p^{\ast}$ is a minimal prime ideal of $L(A)$. By Corollary 8.10, $L(A)$ is a conormal lattice, thus any minimal prime ideal of $L(A)$ is a $\sigma$ - ideal. Therefore  $p^{\ast}$ is a $\sigma$ - ideal of $L(A)$. Since $A$ is semiprime, by applying Proposition 4.7 it follows that $p = (p^{\ast})_{\ast}$ is a pure element of $A$.
\end{proof}

\begin{corolar}Let $A$ be a semiprime quantale. Then $A$ is a $PF$ - quantale if and only if any minimal $m$ - prime element of $A$ is pure.

\end{corolar}

\begin{proof} We apply Theorems 8.19 and 8.20.

\end{proof}

\begin{teorema}
For a coherent quantale $A$ the following are equivalent:
\usecounter{nr}
\begin{list}{(\arabic{nr})}{\usecounter{nr}}

\item $A$ is a $PF$ - quantale;

\item $A$ is a semiprime $mp$ - quantale;

\item If $c,d\in K(A)$ then $cd = 0$ implies $c^{\perp}\lor d^{\perp} = 1$;

\item If $c,d\in K(A)$ then $(cd)^{\perp}$ = $c^{\perp}\lor d^{\perp}$;

\item For each $c\in K(A)$, $c^{\perp}$ is a pure element.

\end{list}
\end{teorema}

\begin{proof} The equivalence of $(1)$ and $(2)$ follows from Theorem 8.19 and that the conditions $(3)$ and $(5)$ are equivalent is obvious.

 $(2)\Rightarrow(3)$ Assume by absurdum that there exist $c,d\in K(A)$ such that  $cd = 0$ and $c^{\perp}\lor d^{\perp} < 1$, hence there exists a minimal $m$ - prime element $p$ such that $p\leq m$. Since $cd = 0$ and $p\in Spec(A)$ we have $c\leq p$ or $d\leq p$. Assume that $c\leq p$, so $p\lor c^{\perp} = 1$ (by Theorem 8.20, the minimal $m$ - prime element $p$ is pure). This contradicts $p\lor c^{\perp}\leq m$, so the implication is proven.

 $(3)\Rightarrow(2)$ First we prove that $A$ is semiprime. Let $c\in K(A)$  such that $c^n = 0$ for some integer $n\geq 1$. Assuming $n > 1$, from $c^{n-1} c = 0$ one gets $(c^{n-1})^{\perp}$ =  $c^{\perp}\lor (c^{n-1})^{\perp}$ = $1$, hence $c^{n-1} = 0$. By using many times this argument one obtains $c = 0$, so $A$ is semiprime.

 Now we shall prove that $A$ is an $mc$ - quantale. According to Theorem 8.14 it suffices to show that for any distinct minimal $m$ - prime elements $p,q$ we have $p\lor q = 1$. Assume that $p,q\in Min(A)$, $p\neq q$, so there exists $d\in K(A)$ with $d\leq p$ and $d\not\leq q$. By Corollary 8.4 we have $d^{\perp}\not\leq p$, so there exists $c\in K(A)$ such that $c\leq d^{\perp}$ and $c\not\leq p$. Thus $cd = 0$ implies $c^{\perp}\lor d^{\perp} = 1$, so there exist $u,v\in K(A)$ such that $u\leq c^{\perp}$, $v\leq d^{\perp}$ and $u\lor v = 1$. Since $p$ is $m$ - prime, from $uc = 0$ and $c\not\leq p$ it results that $u\leq p$. Similarly, one obtains $v\leq q$, so in the both cases we have $p\lor q = 1$.

 $(3)\Rightarrow(4)$ In order to prove that $(cd)^{\perp}\leq c^{\perp}\lor d^{\perp}$ let us consider a compact element $e$ of $A$ such that $e\leq (cd)^{\perp}$. Then $ecd = 0$, so by hypothesis one gets $(ec)^{\perp}\lor (d)^{\perp} = 1$, therefore there exist $u,v\in K(A)$ such that $u\leq (ec)^{\perp}$, $v\leq (d)^{\perp}$ and $u\lor v = 1$. It follows that $eu\leq c^{\perp}$, $ev\leq d^{\perp}$ and $e = e(u\lor v) = eu\lor ev$, hence $e\leq c^{\perp}\lor d^{\perp}$. Then the inequality $(cd)^{\perp}\leq c^{\perp}\lor d^{\perp}$ was proven. The converse inequality is clear, so $(cd)^{\perp}$ = $c^{\perp}\lor d^{\perp}$.

 $(4)\Rightarrow(3)$ It is easy to see that this implication holds.

\end{proof}

\end{document}